\documentclass[11pt]{article}

\usepackage{amsfonts}
\usepackage{amssymb}
\usepackage{amsmath}
\usepackage{color}
\usepackage{amsthm}
\usepackage{dsfont}
\usepackage{graphicx}

 \numberwithin{equation}{section}

 \newtheorem{theorem}{Theorem}[section]
  \newtheorem{proposition}[theorem]{Proposition}
  \newtheorem{corollary}[theorem]{Corollary}
  \newtheorem{lemma}[theorem]{Lemma}
\newtheorem{remark}[theorem]{Remark}
\newtheorem{definition}[theorem]{Definition}

\newcommand{\dx}{\, {\rm d} x}

\newcommand{\ds}{\, {\rm d} s}

\newcommand{\eps}{\varepsilon}

\newcommand{\EE}{\mathbb{E}}
\newcommand{\VV}{\mathbb{V}}

\begin{document}
\title{Further analysis of multilevel Monte Carlo methods for elliptic PDEs with random coefficients}

\author{ A. L. Teckentrup$^\star$, \ R. Scheichl\thanks{Department of Mathematical Sciences, University of Bath, Claverton
Down, Bath BA2 7AY, UK. Email: {\tt R.Scheichl@bath.ac.uk},\
{\tt A.L.Teckentrup@bath.ac.uk} and {\tt E.Ullmann@bath.ac.uk}}, \ M. B. Giles\thanks{Mathematical Institute, University of Oxford, 24-29 St Giles, Oxford OX1 3LB, UK. Email: {\tt Mike.Giles@maths.ox.ac.uk}}, \
 and \ E. Ullmann$^\star$}

\date{}
\maketitle

\begin{abstract}
We consider the application of multilevel Monte Carlo methods to elliptic PDEs with random coefficients. We focus on models of the random coefficient that lack uniform ellipticity and boundedness with respect to the random parameter, and that only have limited spatial regularity.  We extend the finite element error analysis for this type of equation, carried out in~\cite{cst11}, to more difficult problems, posed on non--smooth domains and with discontinuities in the coefficient. For this wider class of model problem, we prove convergence of the multilevel Monte Carlo algorithm for estimating any bounded, linear functional and any continuously Fr\'echet differentiable non--linear functional of the solution. We further improve the performance of the multilevel estimator by introducing level dependent truncations of the Karhunen--Lo\`eve expansion of the random coefficient. Numerical results complete the paper.
\end{abstract}

\noindent{\bf Keywords:} PDEs with stochastic coefficients, log-normal random fields, non--uniformly elliptic, multilevel Monte Carlo, truncated Karhunen-Lo\`eve expansion, output functionals, discontinuous coefficients and corners.

\section{Introduction}

Monte Carlo type methods are widely used in a range of scientific applications. The dimension independent convergence rate of the sampling error makes these methods attractive for high-dimensional problems, which often can not be approximated well by other types of methods. However, even though dimension independent, the convergence rate of conventional Monte Carlo methods is very slow, and getting to high accuracies is often not computationally feasible.

To improve on the convergence of conventional Monte Carlo methods, one can make use of a variety of variance reduction techniques, such as control variates and antithetic sampling. A particular variance reduction technique which has gotten a lot of attention recently, is the multilevel Monte Carlo (MLMC) method. It was first introduced by Heinrich \cite{heinrich} for the computation of high-dimensional, parameter-dependent integrals, and has since been applied in many areas of mathematics related to differential equations. In particular, a lot of research has been done in stochastic differential equations \cite{dh11,giles2,giles1,hsst12,knr11} and several types of partial differential equations (PDEs) with random coefficients \cite{bsz11,cst11,cgst11,gr11,gkss11,graubner}.

In this paper, we are concerned with the application of multilevel Monte Carlo methods to elliptic PDEs with random coefficients.  In particular, we will focus on 
rough coefficients, which cannot be uniformly bounded in the random parameter and only have H\"older continuous trajectories. This type of problem arises, for example, in stochastic groundwater flow modelling, where log-normal random coefficients are frequently used. A fundamental analysis of the multilevel Monte Carlo algorithm applied to this type of model problem was recently done in \cite{cst11}, and also \cite{cgst11} demonstrates numerically the effectiveness of multilevel Monte Carlo methods applied to elliptic PDEs with log--normal coefficients. The purpose of this paper is to extend the analysis to cover more situations of practical interest, and to expand on some of the issues raised in \cite{cst11} and \cite{cgst11}.

The analysis in \cite{cst11} addresses the convergence of the multilevel Monte Carlo method for simple output functionals, e.g. the $L^2$ or the $H^1$ norm of the solution. In practical situations, however, one is often interested in more complicated functionals, such as the outflow through parts of the boundary or the position of particles transported in the flow field. Here, we therefore extend the convergence analysis to cover bounded linear, as well as continuously Fr\'echet differentiable nonlinear functionals of the solution.

Another issue, raised both in \cite{cgst11} and \cite{cst11}, is the influence of the rough nature of our model problem on the performance of the multilevel Monte Carlo estimator. The oscillatory nature and the short characteristic length scale of the random coefficients puts a bound on how coarse the coarsest level in the multilevel estimator can be. Asymptotically, as the required accuracy goes to 0, this does not have any effect on the cost of the MLMC estimator. For a fixed tolerance, however, it restricts the gain that we can expect compared to a standard Monte Carlo estimator.  In this paper, we propose a solution to this problem by using smoother approximations of the random coefficient on the coarse levels of the estimator. This allows us to choose the coarsest level independent of the length scale on which the random coefficient varies. See also \cite{gkss11} for a similar strategy in the context of the related Brinkman problem.  In \cite{gkss11} the decay rate for the FE error with respect to the number of KL-modes $K$ was assumed. Here we make no such assumption and instead use the decay rates established in \cite{charrier,cst11} for certain log-normal fields and covariance functions.

Finally, we extend the theoretical aspects of \cite{cst11} to more challenging model problems. This includes problems posed on polygonal domains, which are frequently used in connection with finite element methods, as well as problems where the random coefficient has a jump discontinuity. It is well known that these type of model problems do not always exhibit full global regularity, which directly influences the convergence rates of the finite element error.

The outline of the rest of the paper is as follows. In \S 2, we present the model problem, together with the main results on its regularity and finite element error estimates. This follows closely the work in \cite{cst11}. The proof of the new regularity result for polygonal /polyhedral domains is postponed to \S 5. For the reader's convenience, we also briefly recall the multilevel Monte Carlo algorithm and the abstract convergence theorem. In \S 3, we prove convergence of the MLMC algorithm for a wide class of (linear and nonlinear) output functionals, including boundary fluxes and local averages of the pressure. In \S 4, we improve on the performance of the MLMC estimator by using smoother approximations of the random coefficient on coarser levels. The gains possible with this approach are verified both theoretically and numerically. Finally, in \S 5, we give a detailed proof of the regularity result stated in \S 2, and extend the results further to certain classes of discontinuous coefficients.

The key task in this paper is to keep track of how the constants in the bounds and 
estimates depend on the coefficient $a(\omega,x)$ and on the mesh size $h$. Hence, we 
will almost always be stating constants explicitly. Constants 
that do not depend on $a(\omega,x)$ or $h$ will not be explicitly stated. Instead, 
we will write $b \lesssim c$ for two positive quantities $b$ and $c$, if $b/c$ is 
uniformly bounded by a constant independent of $a(\omega,x)$ and of $h$.


\section{Background}
\subsection{Problem setting and basic finite element error analysis}
\label{sec:prob}

Given a probability space $\left( \Omega, \mathcal{A}, \mathbb{P}\right)$ and $\omega \in \Omega$, we consider the following linear elliptic partial differential equation (PDE) with random coefficients, posed on a bounded, Lipschitz polygonal/polyhedral domain $D \subset \mathbb{R}^d$, $d=2,3$, and subject to Dirichlet boundary conditions: Find $u:\Omega\times D \to \mathbb{R}$ such that
\begin{align}
\label{mod}
-\nabla \cdot \left(a(\omega, x) \nabla u(\omega, x)\right) &= 
f(\omega,x), \ \qquad \mathrm{for} \ x \in D, \\
u(\omega,x) &= \phi_j(\omega,x), \qquad \mathrm{for} \ x \in \Gamma_j\,.  \nonumber
\end{align}
The differential operators $\nabla \cdot$ and $\nabla$ are with respect to $x \in D$, and $\Gamma:= \cup_{j=1}^m \overline \Gamma_j$ denotes the boundary of $D$, partitioned into straight line segments in 2D and into planar polygonal panels in 3D. We assume that the boundary conditions are compatible, i.e. $\phi_j \equiv \phi_{k}$, if $\overline \Gamma_j \cap \overline \Gamma_{k} \not= \emptyset$. We also let $\phi \in H^1(D)$ be an extension of the boundary data $\{\phi_j\}_{j=1}^m$ to the interior of $D$ whose trace coincides with $\phi_j$ on $\Gamma_j$.

Let us formally define, for all $\omega \in \Omega$,
\begin{equation}
\label{amin}
a_{\mathrm{min}}(\omega):= \min_{x \in \overline{D}} a(\omega, x) \qquad \mathrm{and} 
\qquad a_{\mathrm{max}}(\omega):= \max_{x \in \overline{D}} a(\omega, x).
\end{equation}

We make the following assumptions on the input data:
\begin{itemize}
\item[{\bf A1.}] $a_{\mathrm{min}} \ge 0$ almost surely and 
$1/a_{\mathrm{min}} \in L^{p}(\Omega)$, for all $p \in (0,\infty)$.
\item[{\bf A2.}] $a \in L^{p}(\Omega,\mathcal C^{t}(\overline D))$, for some $0 < t \le 1$ and 
for all $p \in (0,\infty)$.
\item[{\bf A3.}] $f \in L^{p_*}(\Omega,H^{t-1}(D))$ and $\phi_j \in L^{p_*}(\Omega,H^{t+1/2}(\Gamma_j))$, $j=1,\dots,m$, for some $p_* \in (0,\infty]$.
\end{itemize}

Here, the space $\mathcal C^{t}(\overline D)$ is the space of H\"older--continuous functions with exponent $t$, $H^s(D)$ is the usual fractional order Sobolev space, and  $L^q(\Omega, \mathcal B)$ denotes the space of $\mathcal B$-valued random fields, for which the $q^\mathrm{th}$ moment (with respect to the measure $\mathbb P$) of the $\mathcal B$--norm is finite, see e.g \cite{cst11}. A space which will appear frequently in the error analysis is the space $L^q(\Omega, H^1_0(D))$, which denotes the space of $H_0^1(D)$--valued random fields with the norm on $H_0^1(D)$ being the usual $H^1(D)$--seminorm $|\cdot|_{H^1(D)}$. We will weaken Assumption A2 in \S \ref{sec:dis}, and assume only piecewise continuity of $a(\omega, \cdot)$, but chose not to do this here for ease of presentation. For the same reason, we do not choose to weaken Assumptions A1 and A2 to $1/a_{\mathrm{min}}$ and $\|a\|_{\mathcal C^{t}(\overline D)}$ having finite moments of order $p_a$, for some $p_a \in (0,\infty)$, although this is possible. Finally, note that since $a_{\mathrm{max}}(\omega) = \|a\|_{\mathcal C^{0}(\overline D)}$, Assumption A2 implies that $a_{\mathrm{max}} \in L^p(\Omega)$, for any $p \in (0,\infty)$.

To simplify the notation in the following, let $0 < C_{a,f,\phi_j} < \infty$ denote a generic constant which depends algebraically on $L^q(\Omega)$--norms of $a_{\mathrm{max}}, \, 1/a_{\mathrm{min}}, \, \|a\|_{C^{t}(\overline D)}, \, \|f\|_{H^{t-1}(D)}$ and $\|\phi_j\|_{H^{t+1/2}(\Gamma_j)}$, with $q < p*$ in the case of $\|f\|_{H^{t-1}(D)}$ and $\|\phi_j\|_{H^{t+1/2}(\Gamma_j)}$. Two additional random variables related to output functionals will be added to this notation later.

An example of a random field $a(\omega,x)$ that satisfies Assumptions A1 and A2, for all $p \in (0,\infty)$, is a log-normal random field $a = \exp(g)$, where the underlying Gaussian field $g$ has a H\"older--continuous mean and a Lipschitz continuous covariance function. For example, $g$ could have constant mean and an exponential covariance function, given by
\begin{equation}\label{cov:exp}
\mathbb E\Big[(g(\omega,x)-\mathbb E[g(\omega,x)])(g(\omega,y)-\mathbb 
E[g(\omega,y)]) \Big]= \sigma^2 \exp(-\|x-y\|/ \lambda)
\end{equation}
where $\sigma^2$ and $\lambda$ are real parameters known as the {\em variance} and {\em correlation length}, and $\|\cdot \|$ denotes a norm on $\mathbb{R}^d$.
If $\|\cdot \| = \|\cdot \|_p$, we will call it a $p$-norm exponential covariance.

If we denote by $H^1_\phi(D) := \{v \in H^1(D) : v - \phi = 0 \text{ on } \Gamma\}$, then the variational formulation of \eqref{mod}, parametrised by $\omega \in \Omega$, is to find $u \in H^1_\phi(D)$ such that
\begin{equation}\label{weak}
b_\omega\big(u(\omega,\cdot),v\big) = L_\omega(v) \,, 
\quad \text{for all} \quad v \in H^1_0(D).
\end{equation}
The bilinear form $b_\omega$ and the linear functional
$L_\omega$ (both parametrised by $\omega$) are defined as usual, for all 
$u,v \in H^1(D)$, by
\begin{equation}\label{bilinear}
b_\omega(u,v) := \int_D a(\omega,x)\nabla u(x) \cdot \nabla v(x) \dx \quad \text{and} \quad 
L_\omega(v) := \langle f(\omega,\cdot), v\rangle_{H^{t-1}(D),H^{1-t}_0(D)}\,,
\end{equation}
where $H^{1-t}_0(D)$ is the closure of $C_0^\infty(D)$ in the $H^{1-t}(D)$-norm.
We say that for any $\omega \in \Omega$, $u(\omega,\cdot)$ is a 
weak solution of \eqref{mod} iff $u(\omega,\cdot) \in H^1_\phi(D)$ and $u(\omega,\cdot)$ satisfies 
(\ref{weak}). An application of the Lax--Milgram Theorem ensures existence and uniqueness of $u(\omega,\cdot) \in H^1_\phi(D)$, for almost all $\omega$, and in combination with Assumptions A1-- A3, this gives the existence of a unique solution $u \in L^p(\Omega, H^1(D))$, for any $p <p_*$.

In \cite{cst11}, a regularity analysis of the above model problem was performed under the assumptions that the spatial domain $D$ is $\mathcal{C}^2$. Here, the analysis is extended to polygonal domains, or more generally, to piecewise $\mathcal{C}^2$ domains that are rectilinear near the corners. This is very important since in standard finite element methods one naturally works with polygonal/polyhedral domains.

\begin{definition}
\label{def:laplace}\em
Let $0 < \lambda_{\Delta}(D) \le 1$ be such that for any $0 < s \leq \lambda_{\Delta}(D), s \neq \frac{1}{2}$, the Laplace operator $\Delta$ is surjective as an operator from $H^{1+s}(D) \cap H^1_0(D)$ to $H^{s-1}(D)$. In other words, let $\lambda_{\Delta}(D)$ be no larger than the order of the strongest singularity of the Laplace operator with homogeneous Dirichlet boundary conditions on $D$.
\end{definition}

\begin{theorem}
\label{regu} 
Let Assumptions A1-A3 hold for some $0 < t \le 1$. Then,
\begin{equation}
\label{ineq:regu}
 \|u(\omega, \cdot)\|_{H^{1+s}(D)} \lesssim \frac{a_{\mathrm{max}}(\omega)\|a(\omega, \cdot)\|^2_{\mathcal C^{t}(\overline D)}}{a_{\mathrm{min}}(\omega)^4}  \, \left[\|f(\omega, \cdot)\|_{H^{t-1}(D)} + \|a(\omega, \cdot)\|_{\mathcal C^{t}(\overline D)} \sum_{j=1}^m\|\phi_j(\omega, \cdot)\|_{H^{t+1/2}(\Gamma_j)} \right]
\end{equation}
for almost all $\omega \in \Omega$ and for all $0 < s < t$ such that $s \le \lambda_{\Delta}(D)$. Moreover, $u \in L^p(\Omega,H^{1+s}(D))$, for all \;$p < p_*$. If $t=\lambda_{\Delta}(D)=1$, 
then $u \in L^p(\Omega,H^{2}(D))$ and the above bound holds with $s=1$.
\end{theorem}

\begin{proof}
The proof for individual samples, for almost all $\omega \in \Omega$, is a classical result and follows Grisvard \cite[Section 5.2]{grisvard2}. A detailed proof making precise the dependence of the bound on the coefficients is given in Section \ref{sec:poly}. The remainder of the theorem follows by H\"older's inequality from Assumptions A1--A3.
\end{proof}

Theorem \ref{regu} can now be used to prove convergence of finite element approximations of $u$ in the standard way. We will only consider lowest order elements in detail. Introduce a triangulation $\mathcal T_h$ of $D$, and let $V_h$ be the space of continuous, piecewise linear functions on $D$ that satisfy the boundary conditions in \eqref{mod}, i.e.
\[
V_{h,\phi}:=\left\{v_h \in \mathcal{C}(\overline D)\, : \, v_h|_T \, \text{linear, for all } T \in \mathcal T_h, \quad \text{and} \quad v_h|_{\Gamma_j} = \phi_j, \, \text{for all } j=1,\dots,m \right\}.
\]
For simplicity we assume that the functions $\phi_j$, $j=1,\ldots,m$, are piecewise linear with respect to the triangulation $\mathcal T_h$ restricted to $\Gamma_j$. To deal with more general boundary conditions is a standard exercise in finite element analysis (see e.g. \cite[\S 10.2]{brenner_scott}).  

The finite element approximation of $u$ in $V_{h,\phi}$, denoted by $u_h$, is now found by solving 
\[
b_\omega\big(u_h(\omega,\cdot),v\big) = L_\omega(v) \,, 
\quad \text{for all} \quad v \in V_{h,0},
\]
Using Cea's lemma and standard interpolation results on $V_{h,\phi}$, we then have (as in \cite{cst11}) the following.

\begin{theorem}\label{h1fe} Let Assumptions A1-A3 hold for some $0 < t \le 1$. Then,
\[
|(u-u_h)(\omega,\cdot)|_{H^1(D)} \; \lesssim  \; \left(\frac{a_{\mathrm{max}}(\omega)}{a_{\mathrm{min}}(\omega)}\right)^{1/2} \|u(\omega, \cdot)\|_{H^{1+s}(D)} \, h^s
\]
for almost all $\omega \in \Omega$ and for all $0 < s < t$ such that $s \le \lambda_{\Delta}(D)$. Hence, 
\[
\|u-u_h\|_{L^p(\Omega, H^1_0(D))} \; \le  \;  C_{a,f,\phi_j} \, h^s, \qquad \text{for all} \ p < p_* \,,
\]
with $C_{a,f,\phi_j} < \infty$ a constant that depends on the input data, but 
is independent of $h$.
If A1-A3 hold with $t=\lambda_{\Delta}(D)=1$, then $\|u-u_h\|_{L^p(\Omega, H^1_0(D))} \le C_{a,f,\phi_j} \, h$. 
\end{theorem}

The key novel result here (extending the results in \cite{cst11}) is that the rate of convergence of the finite element error on polygonal/polyhedral domains $D$ is the same as on $C^2$ domains provided the order of the strongest singularity for the Laplacian on $D$ is no stronger than $t$ in A1-A3. No additional or stronger singularities are triggered by the random coefficient provided $a$ satisfies A2. A sufficient (but not necessary) condition for $t = 1$ is that $D$ is convex. For $t < 1$, even certain concave domains are allowed. 

\begin{remark} \em
The results can easily be extended also to Neumann and mixed Dirichlet/Neumann boundary conditions. We will comment on this in Section \ref{sec:poly} and confirm it with some of the numerical results in Section \ref{sec:num_funcs}. There is also no fundamental difficulty in extending the analysis to higher order finite elements. For an extension to mixed finite elements see \cite{gsu12}.
\end{remark}


\subsection{Multilevel Monte Carlo Algorithm}
\label{sec:multilevel}
Before we go on to the main part of this paper, we will briefly recall the multilevel Monte Carlo algorithm. We also give a review of the main results on its convergence when applied to elliptic PDEs of the form described in the previous section.

Suppose we are interested in finding the expected value of some functional $Q=M(u)$ of the solution $u$ to our model problem \eqref{mod}. Since $u$ is not easily accessible, $Q$ is often approximated by the quantity $Q_h:=M(u_h)$, where $u_h$ is a finite dimensional approximation to $u$, such as the finite element solution on a sufficiently fine spatial grid $\mathcal{T}_h$ defined above. However, $u_h$ may also include further approximations such as an inexact bilinear form $b^h_\omega(\cdot,\cdot) \approx b_\omega(\cdot,\cdot)$, e.g. due to quadrature or approximation of the input random field $a$. We will return to this issue in Section \ref{sec:level}.

To estimate $\EE\left[Q\right]$,
we then compute approximations (or {\it estimators}) $\widehat{Q}_h$ to
$\EE\left[Q_h\right]$, and quantify the accuracy of our approximations via the root
mean square error (RMSE)
\[
e(\widehat{Q}_h) := \left(\mathbb{E}\big[(\widehat{Q}_h - \mathbb{E}(Q))^2\big]\right)^{1/2}.
\]
The computational cost $\mathcal{C}_\eps(\widehat{Q}_h)$ of our estimator is then
quantified by the number of floating point operations that are needed to achieve a
RMSE of $e(\widehat{Q}_h) \leq \eps$. This will be referred to as the $\eps$--cost.

The classical Monte Carlo (MC) estimator for $\EE\left[Q_h\right]$ is
\begin{equation}
\label{MC}
\widehat{Q}^\mathrm{MC}_{h,N} := \frac{1}{N} \sum_{i=1}^N Q_h(\omega^{(i)}),
\end{equation}
where $Q_h(\omega^{(i)})$
is the $i$th sample of $Q_h$ and $N$ independent samples are computed in total.

There are two sources of error in the estimator~\eqref{MC}, the
approximation of $Q$ by $Q_h$, which is related to the spatial
discretisation, and the sampling error due to replacing the expected value
by a finite sample average. This becomes clear when expanding the mean
square error (MSE) and using the fact that for Monte Carlo
$\EE[\widehat{Q}^\mathrm{MC}_{h,N}] = \EE[Q_h]$ and
$\VV[\widehat{Q}^\mathrm{MC}_{h,N}] = N^{-1} \, \VV[Q_h],$ where $\VV[X] :=
\EE[(X-\EE[X])^2]$ denotes the variance of the random variable $X:\Omega \to \mathbb{R}$.
We get
\begin{equation}
\label{msesd2}
e(\widehat{Q}^\mathrm{MC}_{h,N})^2 \ = \ N^{-1} \VV[Q_h] + \big(\EE[Q_h - Q]\big)^2.
\end{equation}
A sufficient condition to achieve a RMSE of $\eps$ with this estimator is that both of
these terms are less than $\eps^2/2$. For the first term, this is achieved by choosing
a large enough number of samples, $N=\mathcal{O}(\eps^{-2})$. For the second term, we
need to choose a fine enough finite element mesh $\mathcal{T}_h$, such that
$\EE[Q_h - Q] = \mathcal{O}(\eps)$.

The main idea of the MLMC estimator is very simple. We sample not just from one
approximation $Q_h$ of $Q$, but from several. Linearity of the expectation operator
implies that
\begin{equation}
\EE[Q_h] = \EE[Q_{h_0}] + \sum_{\ell=1}^L \EE[Q_{h_\ell} - Q_{h_{\ell-1}}]
\label{eq:identity}
\end{equation}
where $\{h_\ell\}_{\ell = 0, \dots, L}$ are the mesh widths of a sequence of increasingly
fine triangulations $\mathcal{T}_{h_\ell}$ with
$\mathcal{T}_h := \mathcal{T}_{h_L}$, the finest mesh, and $k_1 \leq h_{\ell-1}/h_\ell \le k_2$,
for all $\ell = 1, \dots, L$ and some $1 < k_1 \leq k_2 <\infty$.
Hence, the expectation on the finest mesh is equal to
the expectation on the coarsest mesh, plus a sum of corrections adding
the difference in expectation between simulations on consecutive meshes.
The multilevel idea is now to independently estimate each of these
terms such that the overall variance is minimised for
a fixed computational cost.

Setting for convenience $Y_0 := Q_{h_0}$ and $Y_\ell := Q_{h_\ell} - Q_{h_{\ell-1}}$, for
$1 \leq \ell \leq L$, we define the MLMC estimator simply as
\begin{equation}
\widehat{Q}^\mathrm{ML}_{h,\{N_\ell\}} \ := \ \sum_{\ell=0}^L \widehat{Y}^\mathrm{MC}_{\ell, N_\ell} \ = \ \sum_{\ell=0}^L \frac{1}{N_{\ell}} \sum_{i=1}^{N_{\ell}} Y_\ell(\omega^{(i)}),
\end{equation}
where importantly $Y_\ell(\omega^{(i)}) = Q_{h_\ell}(\omega^{(i)}) -
Q_{h_{\ell-1}}(\omega^{(i)})$, i.e.~using the same sample on both meshes.

Since all the expectations $\EE[Y_{\ell}]$ are estimated independently in
\eqref{eq:identity}, the variance of the MLMC estimator is
$
\sum_{\ell=0}^L N_{\ell}^{-1}\,\VV[Y_\ell]
$
 and expanding as in~(\ref{msesd2}) leads again to 
\begin{equation}\label{mseml}
e(\widehat{Q}^\mathrm{ML}_{h,\{N_\ell\}})^2 \; := \;
\mathbb{E}\Big[\big(\widehat{Q}^\mathrm{ML}_{h,\{N_\ell\}} - \mathbb{E}[Q]\big)^2\Big] \; = \;
\sum_{\ell=0}^L N_{\ell}^{-1}\,\VV[Y_{\ell}] \;+ \; \big(\mathbb{E}[Q_{h} - Q]\big)^2.
\end{equation}
As in the classical MC case before, we see that the MSE consists of two terms,
the variance of the estimator and the error in mean between $Q$ and $Q_h$. Note that
the second term is identical to the second term for the classical MC method in
\eqref{msesd2}. 

Let now $\mathcal{C}_\ell$ denote the cost to obtain one sample of $Q_{h_\ell}$. Then
we have the following results on the  $\eps$--cost of the MLMC estimator
(cf.~\cite{cgst11,giles1}).

\begin{theorem}
\label{main_thm}
Suppose that
there are positive constants $\alpha, \beta, \gamma, c_{\scriptscriptstyle \mathrm{M1}}, c_{\scriptscriptstyle \mathrm{M2}}, c_{\scriptscriptstyle \mathrm{M3}} > 0$ such that
$\alpha\!\geq\!\frac{1}{2}\,\min(\beta,\gamma)$ and
\begin{itemize}
\item[{\bf M1.}]
$\displaystyle
\left| \EE[Q_{h} - Q] \right| \  \leq c_{\scriptscriptstyle \mathrm{M1}} \ h^{\alpha},
$
\item[{\bf M2.}]
$\displaystyle
\VV[Q_{h_\ell}-Q_{h_{\ell-1}}] \ \leq c_{\scriptscriptstyle \mathrm{M2}} \ h_\ell^{\beta},
$
\item[{\bf M3.}]
$\displaystyle
\mathcal{C}_{\ell} \ \leq c_{\scriptscriptstyle \mathrm{M3}} \ h_{\ell}^{-\gamma},
$
\end{itemize}
Then, for any $\;\eps < e^{-1}$, there exist an $L$ and a sequence
$\{N_{\ell}\}_{\ell=0}^L$, such that $e(\widehat{Q}^\mathrm{ML}_{h, \{N_\ell\}}) < \eps$ and
\[
\mathcal{C}_\eps(\widehat{Q}^\mathrm{ML}_{h, \{N_\ell\}}) \ \lesssim \ \left\{\begin{array}{ll}
\ \eps^{-2}              ,    & \text{if } \ \beta>\gamma, \\[0.02in]
\ \eps^{-2} (\log \eps)^2,    & \text{if } \ \beta=\gamma, \\[0.04in]
\ \eps^{-2-(\gamma - \beta)/\alpha}, & \text{if } \ \beta<\gamma,
\end{array}\right.
\]
where the hidden constant depends on $c_{\scriptscriptstyle \mathrm{M1}}, c_{\scriptscriptstyle \mathrm{M2}}$ and $c_{\scriptscriptstyle \mathrm{M3}}$.
\end{theorem}

In \cite{cst11}, it was shown that for our model problem and for the functional $Q :=|u|^q_{H^1(D)}$ it follows immediately from the finite element error result in Theorem \ref{h1fe} that Assumptions M1--M2 in Theorem \ref{main_thm} hold with $\alpha < t$ and $\beta < 2t$ and for $1 \le q < p_*/2$ provided Assumptions A1--A3 hold with $0 < t < 1$ and $p_* \in (0,\infty)$. For $t=1$, we can even choose $\alpha = 1$ and $\beta = 2$. Using a duality argument we also showed that under the same hypotheses we could expect twice these rates for the functional $Q:=\|u\|^q_{L^2(D)}$. 
The aim is now to extend this theory to cover also other functionals of the solution $u$ (see \S\ref{sec:func}) as well as level-dependent estimators (see \S\ref{sec:level}).


\section{Output functionals}
\label{sec:func}
In practical applications, one is often interested in the expected values of certain functionals of the solution. In the context of groundwater flow modelling, this could for example be the value of the pressure or the Darcy flux at or around a given point in the computational domain, or the outflow over parts of the boundary. It could also be something more complicated, such as positions and travel times of particles released somewhere in the computational domain (see e.g. \cite{gknss11}).

A standard technique to prove convergence for finite element approximations of output functionals is to use a duality argument, similar to the classic Aubin-Nitsche trick used to prove optimal convergence rates for the $L^2(D)$-norm. The specific boundary conditions and forcing terms of the dual problem will depend on the output functional considered. A further advantage of using a duality argument to prove convergence of output functionals, is that the analysis can be used as a starting point for further developments, such as adaptively refined meshes and adjoint error correction (\cite{gp04,giles_suli}). These are areas we aim to explore further in the future.

We will in the following consider both linear and non-linear functionals. We denote the functional of interest by $M_\omega(v)$, for $v \in H^1(D)$. Like the bilinear form $b_\omega(\cdot,\cdot)$, the functional $M_\omega(\cdot)$ is again parametrised by $\omega$, and the analysis is done almost surely in $\omega$. When the functional does not depend on $\omega$, we will simply write $M(\cdot)$ instead of $M_\omega(\cdot)$. Our analysis follows mainly \cite{giles_suli}.


\subsection{Linear functionals}
\label{sec:lin_func}

Since it is simpler, we will first look at linear functionals. Let us assume for the moment that $M_\omega:H^1(D) \to \mathbb{R}$ is linear and bounded on $H^1_0(D)$, i.e.
$M_\omega(v) \lesssim  \|v\|_{H^{1}(D)}, \ \text{for all} \ v \in H^1_0(D).$
Now, let us associate with our {\it primal problem} \eqref{weak} the following {\it dual problem}: find $z(\omega,\cdot) \in H^1_0(D)$ such that 
\begin{equation}\label{weakdual}
b_\omega\big(v, z(\omega,\cdot)\big) = M_\omega(v)\,, 
\quad \text{for all} \quad v \in H^1_0(D),
\end{equation}
and denote by $z_h(\omega,\cdot) \in V_{h,0}$ the finite element approximation to \eqref{weakdual}. We can again apply the Lax-Milgram Theorem to ensure existence and uniqueness of a weak solution $z(\omega,\cdot) \in H^1_0(D)$, for almost all $\omega$. Moreover, since $b_\omega(\cdot,\cdot)$ is symmetric, we will also be able to apply Theorems \ref{regu} and \ref{h1fe}. However, first we make the following observation.

\begin{lemma} Let $M_\omega:H_0^1(D) \to \mathbb{R}$ be linear and bounded. Then,
for almost all $\omega \in \Omega$,  
\begin{equation}\label{funcer}
\left|M_\omega\left(u(\omega,\cdot)\right) - M_\omega\left(u_h(\omega,\cdot)\right)\right| \leq a_{\mathrm{max}}(\omega) \, |u(\omega,\cdot)-u_h(\omega,\cdot)|_{H^1(D)} \, |z(\omega,\cdot)-z_h(\omega,\cdot)|_{H^1(D)} \,. 
\end{equation}
\end{lemma}
\begin{proof}
Dropping for brevity the dependence of the FE functions on $\omega$ and using the linearity of $M_\omega$, the dual problem \eqref{weakdual}, as well as Galerkin orthogonality for the primal problem, we have
\begin{align*}
\left|M_\omega(u) - M_\omega(u_h) \right|&=\left| b_\omega(u - u_h, z)\right| 
=\left| b_\omega(u - u_h, z - z_h)\right|
\leq a_{\mathrm{max}}(\omega) \, |u-u_h|_{H^1(D)} \, |z-z_h|_{H^1(D)}\,, 
\end{align*}
where in the last step we have used the definition of $a_{\mathrm{max}}(\omega)$ and the Cauchy-Schwarz inequality.
\end{proof}

This simple argument will be crucial to obtain optimal convergence rates for functionals. Provided the functional is bounded not just in $H^1(D)$, but also in $H^{1-t_*}(D)$, for some $t_* \ge t$, where $t$ is as in Assumptions A1--A3, the finite element solution $z_h$ of the dual problem will converge with any order $s < t$ (like the primal solution $u_h$). This will allow us (as for the $L^2(D)$-norm) to obtain a convergence rate twice that of the $H^1(D)$-norm in Theorem \ref{h1fe}. However, the assumption that $M_\omega$ is linear is not necessary, and so we will first generalise the above to nonlinear functionals.


\subsection{Nonlinear functionals}
\label{sec:non_func}

For nonlinear functionals, the dual problem in \cite{gp04} is not defined as in \eqref{weakdual} above. Instead, a different functional is chosen on the right hand side (which reduces to $M_\omega$ in the linear case).
It is related to the derivative of the functional of interest and so we need to assume a certain differentiability of $M_\omega$. We will assume here that $M_\omega$ is continuously Fr\'echet differentiable. In particular, this implies that $M_\omega$ is also Gateaux differentiable, with the two derivatives being the same. We will see in Remark \ref{slant} below that it is in fact not necessary that $M_\omega$ is continuously Fr\'echet differentiable everywhere, but it simplifies the presentation greatly.

Let $v, \tilde{v} \in H^1(D)$. Then the Gateaux derivative of $M_\omega$ at $\tilde v$ and in the direction $v$ is defined as
\[
D_v M_\omega(\tilde v) := \lim_{\varepsilon \rightarrow 0} \frac{M_\omega(\tilde v+\varepsilon v)-M_\omega(\tilde v)}{\varepsilon}.
\] 
We define
\[
\overline{D_v M_\omega}(u, u_h) := \int_0^1 D_v M_\omega(u+\theta(u_h-u)) \,\mathrm{d} \theta,
\]
which is in some sense an average derivative of $M_\omega$ on the path 
from $u$ to $u_h$, and define the dual problem now as: find $z(\omega,\cdot) \in H^1_0(D)$ such that
\begin{equation}\label{weakdual2}
b_\omega\big(v, z(\omega,\cdot)\big) = \, \overline{D_v M_\omega}(u,u_h),
\quad \text{for all} \quad v \in H^1_0(D).
\end{equation}
Note that, for any linear functional $M_\omega$, we have
$\overline{D_v M_\omega}(u,u_h) = M_\omega(v)$, for all $v \in H^1_0(D)$, and so \eqref{weakdual2} is equivalent to \eqref{weakdual}. 

For our further analysis, we need to make the following assumption on $M_\omega$.
\begin{itemize}
\item[{\bf F1.}] 
Let $u$ (resp. $u_h$) be the exact (resp. the FE) solution of 
\eqref{mod}. Let $M_\omega$ be continuously Fr\'echet differentiable, and suppose that there exists $t_* \in [0,1]$, $q_*\in(0,\infty]$ and $C_\mathrm{F} \in L^{q_*}(\Omega)$, such that 
\[
|\overline{D_v M_\omega}(u, u_h)| \lesssim C_\mathrm{F}(\omega)  \|v\|_{H^{1-t_*}(D)}\,,
\qquad \text{for all} \ \ v \in H^1_0(D) \ \ \text{and for almost all} \ \ \omega \in \Omega.
\]
\end{itemize} 
To get well--posedness of the dual problem, as well as existence and uniqueness of the dual solution $z(\omega,\cdot) \in H^1_0(D)$, for almost all $\omega\in \Omega$, it would have been sufficient (as in the linear case courtesy of the Lax-Milgram Theorem) to assume that $|\overline{D_v M_\omega}(u, u_h)|$ is bounded in $H^1(D)$. However, in order to apply Theorem \ref{h1fe} and to prove convergence of the finite element approximation of the dual solution, it is necessary to require stronger spatial regularity for $z$. This is only possible if we assume boundedness of $|\overline{D_v M_\omega}(u, u_h)|$ in $H^{1-t_*}(D)$ for some $t_* > 0$. In particular, if Assumptions A1--A3 and F1 are satisfied with $t \in (0,1]$ and $t_* \ge t$, then for almost all $\omega \in \Omega$,
\begin{equation*}
 \|z(\omega, \cdot)\|_{H^{1+s}(D)} \lesssim \frac{a_{\mathrm{max}}(\omega)\|a(\omega, \cdot)\|^2_{\mathcal C^{t}(\overline D)}}{a_{\mathrm{min}}(\omega)^4}  \, C_\mathrm{F}(\omega),
\end{equation*}
for any $0 < s < t$ such that $s \le \lambda_{\Delta}(D)$ and for almost all $\omega \in \Omega$. Hence, 
\begin{equation}\label{regz}
\|z-z_h\|_{L^p(\Omega, H^1_0(D))} \; \le \; C_{a,C_F} \, h^s, \quad \text{for all} \ \  p < q_*\,,
\end{equation}
for some constant $C_{a,C_F} < \infty$ depending on $a$ and the constant $C_F$ in F1.

Moreover, from the Fundamental Theorem of Calculus for Fr\'echet derivatives, it follows that
\begin{align}\label{eq:mdiff}
M_\omega(u) - M_\omega(u_h) &=  \int_0^1 D_{u-u_h} M_\omega(u+\theta(u_h-u)) \,\mathrm{d} \theta = \overline{D_{u-u_h} M_\omega}(u, u_h) = b_\omega(u - u_h, z)
\end{align}
and so we have again the following error bound.

\begin{lemma} \label{lem:nonlin}
Let Assumption F1 be satisfied, then 
\begin{equation}\label{funcer3}
\left|M_\omega\left(u(\omega,\cdot)\right) - M_\omega\left(u_h(\omega,\cdot)\right)\right| \leq a_{\mathrm{max}}(\omega) \, |u(\omega,\cdot)-u_h(\omega,\cdot)|_{H^1(D)} \, |z(\omega,\cdot)-z_h(\omega,\cdot)|_{H^1(D)} \,,
\end{equation}
for almost all $\omega \in \Omega$.
\end{lemma}

Similar to the bound in \eqref{funcer3}, one can also find a bound of the error between two finite element approximations in $V_h$ and in $V_H \subset V_h$, namely
\begin{equation}
\label{funcer2}
\left|M_\omega\left(u_h(\omega,\cdot)\right) - M_\omega\left(u_H(\omega,\cdot)\right)\right| \leq a_{\mathrm{max}}(\omega) \, |u_h(\omega,\cdot)-u_H(\omega,\cdot)|_{H^1(D)} \, |z_h(\omega,\cdot)-z_H(\omega,\cdot)|_{H^1(D)}  
\end{equation}
In the next section, we will use \eqref{funcer3} and \eqref{funcer2} to find optimal rates for (non)linear functionals in Assumptions M1 and M2 of the MLMC convergence theorem. 

\begin{remark}\label{slant} {\em As already mentioned above, continuous Fr\'echet differentiability is not a necessary condition. It is possible to weaken Assumption F1 and to assume only slant differentiability of $M_\omega(\cdot)$. The concept of slant differentiability was introduced in \cite{cnq01}, where it was also shown that an operator $F:X \rightarrow Y$, for two Banach spaces $X$ and $Y$, is slant differentiable 
iff it is Lipschitz continuous. 
Most importantly, however, slant differentiability is sufficient for proving \eqref{eq:mdiff}, and thus Lemma \ref{lem:nonlin}.} 
\end{remark}

\subsection{Multilevel Monte Carlo convergence for functionals}
\label{sec:funcmlmc}

We are now ready to prove optimal convergence rates for the MLMC algorithm for Fr\'echet differentiable (and thus also for linear) functionals as defined above. In order to apply Theorem \ref{main_thm}, we need bounds on the following two quantities:
\begin{enumerate}
\item[(i)] $\left|\EE \left[M_\omega(u) - M_\omega(u_h)\right]\right|$
\item[(ii)] $\VV \left[M_\omega(u_{h_\ell}) - M_\omega(u_{h_{\ell -1}})\right]$
\end{enumerate}

Using the finite element error analysis in Theorem \ref{h1fe} together with the bounds in Lemma \ref{lem:nonlin} and in Equation \eqref{funcer2}, we are able to derive the following bounds for the convergence rates with respect to $h$ for (i) and (ii).

\begin{proposition}\label{funcconv} Let Assumptions A1--A3 hold for some $0 < t \leq 1$ and $p_* > 2$, and let $M_\omega(\cdot)$ 
satisfy Assumption F1 with $t_* \ge t$ and $q_* > \frac{2p_*}{p_*-2}$.
Then Assumptions M1--M2 in Theorem \ref{main_thm} hold for any $\alpha < 2t$ 
and $\beta < 4t$. For $t=1$, we can choose $\alpha=2$ and $\beta=4$.
\end{proposition}
\begin{proof}
Using Lemma \ref{lem:nonlin} and H\"older's inequality, we have 
\begin{align}\label{eq:mexp}
\left|\EE \left[M_\omega(u) - M_\omega(u_h)\right]\right| 
\leq \|a_{\mathrm{max}}\|_{L^{p_1}(\Omega)} \|u-u_h\|_{L^{p_2}(\Omega, H^1_0(D))} \|z -z_h\|_{L^{p_3}(\Omega, H^1_0(D))} 
\end{align}
where $\sum_{i=1}^3 p_i^{-1} =1$. All norms on the right hand side 
are finite, if we choose  $p_1 < \infty$, 
$p_2<p_*$, and $p_3<q_*$, which is possible if $p_*^{-1} + q_*^{-1} < 1$, in particular if $p_* > 2$ and $q_* > \frac{2p_*}{p_*-2}$. In the case $t < 1$, it then follows from \eqref{regz} and Theorem \ref{h1fe} that
\[
\left|\EE \left[M_\omega(u) - M_\omega(u_h)\right]\right| \; \lesssim \; C_{a,f,\phi_j, C_F} \; h^\alpha, \qquad \text{for any} \ \ \alpha < 2t.
\]  
Similarly, using \eqref{funcer2} and H\"older's inequality, we have 
\begin{align}\label{eq:mvar}
\VV \big[M_\omega(u_{h_\ell}) - &M_\omega(u_{h_{\ell -1}})\big]
\leq \EE \left[ |M_\omega(u_{h_\ell}) - M_\omega(u_{h_{\ell -1}})|^2  \right] \nonumber \\
&\leq \|a^2_{\mathrm{max}}\|_{L^{p_1}(\Omega)} \|\left(u_{h_\ell}-u_{h_{\ell-1}}\right)^2\|_{L^{p_2}(\Omega, H^1_0(D))} \|\left(z_{h_\ell} -z_{h_{\ell-1}}\right)^2\|_{L^{p_3}(\Omega, H^1_0(D))} \nonumber \\
&= \|a_{\mathrm{max}}\|^2_{L^{2p_1}(\Omega)} \|u_{h_\ell}-u_{h_{\ell-1}}\|^2_{L^{2p_2}(\Omega, H^1_0(D))} \|z_{h_\ell}-z_{h_{\ell-1}}\|^2_{L^{2p_3}(\Omega, H^1_0(D))}
\end{align}
where $\sum_{i=1}^3 p_i^{-1} =1$. Again, the norms on the right hand side of \eqref{eq:mvar} are finite, if we choose $p_1 < \infty$, 
$p_2<p_*/2$, and $p_3<q_*/2$, which is possible due to our assumptions that $p_* > 2$ and $q_* > \frac{2p_*}{p_*-2}$. In the case $t < 1$, it follows again from \eqref{regz} and Theorem \ref{h1fe} that
\[
\VV \left[M_\omega(u_{h_\ell}) - M_\omega(u_{h_{\ell-1}})\right] \; \lesssim \; C_{a,f,\phi_j, C_F} \; h^\beta, \qquad \text{for any} \ \ \beta < 4t.
\] 
The slightly faster rates of $\alpha=2$ and $\beta = 4$, for $t=1$, can be proved analogously. 
\end{proof}

\begin{remark}\em\label{rem:quad}
In practice, it is in general necessary to use quadrature to compute the integrals in the bilinear form $b_\omega(v,w)$, thus leading to approximate, mesh-dependent bilinear forms. As a consequence we will compute only an approximate finite element solution $\tilde{u}_h \in V_h$ and Galerkin orthogonality for the primal problem is lost. In general, it is then only possible to prove 
\[
\left|M_\omega\left(u(\omega,\cdot)\right) - M_\omega\left(\tilde u_h(\omega,\cdot)\right)\right| \leq C_F(\omega) \, \|u(\omega,\cdot)-\tilde u_h(\omega,\cdot)\|_{H^1(D)},
\]
instead of \eqref{funcer3}, where $C_F$ is the constant from Assumption F1. It is in fact sufficient that F1 holds with $t_* = 0$ in this case. Consequently, it is only possible to verify Assumptions M1--M2 in Theorem \ref{main_thm} for $\alpha < t$ and $\beta < 2t$ in the case $t<1$. Similarly, we can only prove M1--M2 with $\alpha=1$ and $\beta=2$, if $t=1$. The higher rates of convergence from Proposition~\ref{funcconv} can be recovered, also in the presence of quadrature error, if the coefficient function 
has additional regularity, i.e. if $a(\omega, \cdot) \in \mathcal C^{r}(\overline D)$, with $r \ge 2t$. For an example of a log--normal random field which has this additional regularity, see \S \ref{sec:trunc}.

One can also generalise the results in this section to the case where the dual solution has less spatial regularity than the primal solution. For example, if F1 holds only
for some $t_* \in [0, t)$, Assumptions M1--M2 in Theorem \ref{main_thm} can still be verified, for any $\alpha < t+t_*$ and $\beta<2(t+t_*)$.
\end{remark}


\subsection{Examples of output functionals}\label{sec:funcex}
Before we go on to show some numerical results, we give some examples of output functionals which fit into the framework of \S \ref{sec:lin_func}-\ref{sec:funcmlmc}. We start with linear functionals.
{\renewcommand{\labelenumi}{(\alph{enumi})}
\begin{enumerate}
\item {\bf Point evaluations of pressure:} Since $a(\omega, \cdot) \in \mathcal C^t(\overline D) \subset \mathcal C(\overline D)$, we know that trajectories of the solution $u$ are in $\mathcal C ^{1}(\overline D)$  (see e.g. \cite{GT}), and it is meaningful to consider point values. Consider $M^{(1)}(u) := u(x^*)$, for some $x^* \in D$. For $D \subset \mathbb{R}$, i.e. in one space dimension, we have the compact embedding $H^{1/2 + \delta}(D) \hookrightarrow \mathcal C^{\delta}(\overline D)$, for any $\delta > 0$,
and so  
\begin{equation*}
M^{(1)}(v) = v(x^*) \; \leq \; \|v\|_{\sup} \; \lesssim \; \|v\|_{H^{1/2 + \delta}(D)},
\quad \text{for all} \ \ v \in H^1(D).
\end{equation*}
Hence, Assumption F1 is satisfied for any $t_* < \min(\frac{1}{2},t)$ with $C_\mathrm{F}=1$ and $q_*=\infty$.

 In space dimensions higher than one, point evaluation of the pressure $u$ is not a bounded functional on $H^1_0(D)$. One often regularises this type of functional by approximating the point value by a local average, 
\begin{equation*}
M^{(2)}(v) := \frac{1}{|D^*|} \int_{D^*} v(\omega, x) \dx  \quad \Big[\,\approx \; v(\omega,x^*)\,\Big] 
\end{equation*}
where $D^*$ is a small subdomain of $D$ that contains $x^*$ \cite{giles_suli}. Here, $M^{(2)}$ satisfies F1 with $C_\mathrm{F}=1$, $t_* = 1$ and $q_* = \infty$, due to the Cauchy-Schwarz inequality. 

Similarly, point evaluations of the flux $-a \nabla u$ can be approximated by a local average. However, in this case F1 only holds for $t_* = 0$ with $C_\mathrm{F}=a_{\max}$ and $q_* = \infty$, and the convergence rate thus is the same as for the $H^1$-seminorm.
\end{enumerate}

Next we give some examples of non--linear functionals. The first obvious example is to estimate higher order moments of linear functionals.

\begin{enumerate}
\setcounter{enumi}{1}
\item{\bf Second moment of average local pressure:} Let $M_\omega$ be an arbitrary linear functional and let $q > 1$. Then 
\begin{align*}
D_v \big(M_\omega(\tilde v)^q\big) &= \lim_{\varepsilon \rightarrow 0} \frac{M_\omega(\tilde v+\varepsilon v)^q - M_\omega(\tilde v)^q}{\varepsilon} \\
&= \lim_{\varepsilon \rightarrow 0} \frac{\left(M_\omega(\tilde v)+\varepsilon M_\omega(v)\right)^q - M_\omega(\tilde v)^q}{\varepsilon}
\; = \; q M_\omega(\tilde v)^{q-1} M_\omega(v).
\end{align*}
Thus, in case of the second moment of the  average local pressure $M_\omega^{(3)}(v) := \Big(M^{(2)}(v)\Big)^2$, this gives
\[
D_v M^{(3)}_\omega(\tilde v) = \frac{2}{|D^*|^2} \left(\int_{D^*} v(x) \dx \right) \, \left(\int_{D^*} \tilde v(x) \dx \right),
\]
and so
\begin{align*}
|\overline{D_v M^{(3)}_\omega}(u, u_h)| &=  \frac{2}{|D^*|^2} \left|\left(\int_{D^*} v(x) \dx \right) \, \left(\int_0^1 \int_{D^*} (u + \theta(u_h-u))(x) \dx \mathrm{d} \theta\right)\right| \\
&= \frac{1}{|D^*|^2} \left|\left(\int_D v(x) \dx \right) \, \left(\int_D (u(\omega,x)+u_h(\omega,x)) \dx \right)\right| \\[1ex]
& \lesssim \underbrace{\left( \|u(\omega, \cdot)\|_{L^2(D)} + \|u_h(\omega, \cdot)\|_{L^2(D)} \right)}_{=:C_\mathrm{F}(\omega)} \|v\|_{L^2(D)}\,.
\end{align*}

Now, it follows from the Lax-Milgram Theorem that 
$C_\mathrm{F}(\omega) \lesssim \|f(\omega,\cdot)\|_{H^{-1}(D)} / a_\mathrm{min}(\omega)$,
and so Assumption F1 is satisfied for all $t_* \le 1$ and $q_* < p_*\,$.

\item {\bf Outflow through boundary: } Consider $M^{(4)}_{\omega}(v) := L_\omega(\psi) - b_\omega(\psi,v)$, for some given function $\psi \in H^{1}(D)$. Note that for the solution $u$ of \eqref{weak}, by Green's formula, we have
\begin{align}
\label{flux}
M^{(4)}_{\omega}(u) &= \int_D \psi(x) f(x,\omega) \dx - \int_D a(\omega,x) \nabla \psi(x) \cdot \nabla u(\omega,x) \dx  \nonumber \\
&= -\int_D \psi(x) \, \nabla \cdot \left(a(\omega, x) \nabla u(\omega, x)\right)  \dx - \int_D a(\omega,x) \nabla \psi(x) \cdot \nabla u(\omega,x) \dx \nonumber \\
&= - \int_{\Gamma} \psi(x) a(\omega,x) \nabla u(\omega,x) \cdot \nu \ds \,. 
\end{align}
Thus, $M^{(4)}_{\omega}(u)$ is equal to the outflow through the boundary $\Gamma$ weighted by the function $\psi$, and so $M^{(4)}$ can be used to approximate the flux through a part $\Gamma_\mathrm{out} \subset \Gamma$ of the boundary, by setting $\psi |_{\Gamma_\mathrm{out}} \approx 1$ and $\psi |_{\Gamma \backslash \Gamma_\mathrm{out}} \approx 0$, see e.g. \cite{barrett,ddw74,giles_suli}. 

Note that for $f \not\equiv 0$ this functional is only affine, not linear. When $f \equiv 0$, then it is linear.
In any case, 
\begin{align*}
D_v M^{(4)}_\omega(\tilde v) &:=\lim_{\varepsilon \rightarrow 0} \frac{M^{(4)}_\omega(\tilde v+\varepsilon v)-M^{(4)}_\omega(\tilde v)}{\varepsilon} 
\; =\;\lim_{\varepsilon \rightarrow 0} \frac{- \int_D a(\omega,x) \nabla \psi(x) \cdot \nabla (\varepsilon v(\omega,x)) \dx}{\varepsilon} \\
&=  - \int_D a(\omega,x) \nabla \psi(x) \cdot \nabla v(x) \dx 
\;= \; \int_D v(x) \, \nabla \cdot \left(a(\omega, x) \nabla \psi(x)\right)  \dx\,,
\end{align*}
for $v, \tilde{v} \in H^1_0(D)$.
Since this is independent of $\tilde v$, we have in particular
\[
\overline{D_v M^{(4)}_\omega}(u, u_h) = \int_D v(x) \, \nabla \cdot \left(a(\omega, x) \nabla \psi(x)\right)  \dx. 
\]

If we now assume that Assumptions A1-A3 are satisfied for some $0< t \le 1$ and that $\psi \in H^{1+t}(D)$, then using Theorems 9.1.12 and 6.2.25 in \cite{hackbusch} (see also Lemmas A.1 and A.2 in \cite{cst11}), we have $\nabla \psi \in H^{t}(D)$ and for any $t^* < t$,
\begin{align}\label{eq_outbound}
|\overline{D_v M^{(4)}_\omega}(u, u_h)| &\leq \|\nabla \cdot \left(a(\omega, \cdot) \nabla \psi\right)\|_{H^{t^*-1}(D)} \|v\|_{H^{1-t^*}(D)} \nonumber \\
&\lesssim \|\left(a(\omega,\cdot) \nabla \psi\right)\|_{H^{t^*}(D)} \|v\|_{H^{1-t^*}(D)} \nonumber \\
&\lesssim \|a(\omega, \cdot)\|_{\mathcal C^t(\overline D)} \| \nabla \psi \|_{H^{t^*}(D)}\|v\|_{H^{1-t^*}(D)}.
\end{align}
Hence, Assumption F1 is satisfied, for any $q_* < \infty$ and $t_* < t$, with $C_\mathrm{F}(\omega)=\|a(\omega, \cdot)\|_{\mathcal C^t(\overline D)}$. If $t=1$, then estimate \eqref{eq_outbound} holds with $t^*=t=1$, and Assumption F1 is satisfied with $t_*=1$. Our assumption on $\psi$ is satisfied for example if $\psi$ is linear, which is a suitable choice for the numerical test in the next section.
\end{enumerate}

Note that the functional $\frac{1}{\Gamma_{\mathrm{out}}} \int_{\Gamma_{\mathrm{out}}} a(\omega,x) \nabla u(\omega,x) \cdot \nu \ds$ (or rather its regularised equivalent over a narrow region near $\Gamma_{\mathrm{out}}$) can only be bounded in $H^1(D)$ and thus it will converge with a slower rate than $M^{(5)}_\omega$.

\subsection{Numerics}
\label{sec:num_funcs}
We consider two different model problems in 2D, both in the unit square 
$D=(0,1)^2$: either \eqref{mod} with $f \equiv 1$ and $\phi \equiv 0$, i.e. 
\begin{align}
\label{modnum1}
-\nabla \cdot \left(a(\omega, x) \nabla u(\omega, x)\right) = 1, \ \ \mathrm{for} \ x \in D,\quad \mathrm{and} \ \ u(\omega,x) = 0 \ \ \mathrm{for} \ \ x \in \partial D,
\end{align}
or the mixed boundary value problem
\begin{align}
\label{modnum2}
- \nabla \cdot \left(a(\omega, x) \nabla u(\omega, x)\right) &= 0, \qquad &\mathrm{for} \ \ x \in D, \\
u\big|_{x_1=0}= 1, \quad  u\big|_{x_1=1}=0, \quad \frac{\partial u}{\partial \nu}\Big|_{x_2=0} \,&=0, \quad  \frac{\partial u}{\partial \nu}\Big|_{x_2=1} =0. \nonumber
\end{align}
We take $a(\omega,x)$ to be a log-normal random field with exponential covariance function (using the 2-norm in \eqref{cov:exp}) and the underlying Gaussian field has mean zero. We choose $\lambda=0.3$ and $\sigma^2=1$. The finite element solutions are computed on a family of uniform triangular grids $\mathcal{T}_h$ with mesh widths $h = 1/2, 1/4, \ldots, 1/128$. The sampling from $a(\omega,x)$ is done using a circulant embedding technique (for details see \cite{dn97,gknss11}). To assemble the stiffness matrix we have to use a quadrature rule. We chose the trapezoidal rule, evaluating the coefficient function at the vertices of the grids. 

First, we consider the approximation of the pressure at the centre of the domain for model problem \eqref{modnum1}. As described in \S \ref{sec:funcex} for functional $M^{(2)}$, we approximate it by the average of $u_h$ over the region $D^*$, which is chosen to consist of the six elements (of a uniform grid with $h^*=1/256$) adjacent to the node at $(1/2, 1/2)$. To estimate the errors we approximated the exact solution $u$ by a reference solution $u_{h^*}$ on a grid with mesh width $h^*=1/256$. In Figure \ref{fig:midp}, we see that $\left|\EE\left[M^{(2)}(u_{h^*})-M^{(2)}(u_{h})\right]\right|$ converges linearly in $h$ and  $\VV\left[M^{(2)}(u_{h})-M^{(2)}(u_{2h})\right]$ converges quadratically, as predicted by Lemma~\ref{funcconv} for the ``exact'' FE solution. However, in the context of numerical quadrature this is better than expected (cf. Remark \ref{rem:quad}). This suggests that the quadrature error is not dominant here. In Figure~\ref{fig:midp2}, we estimate the second moment of the same functional, and see that also in this case we observe the convergence rates predicted by Lemma \ref{funcconv}.

\begin{figure}[t!]
\centering
\hspace*{-0.75cm}\includegraphics[width=0.5\textwidth]{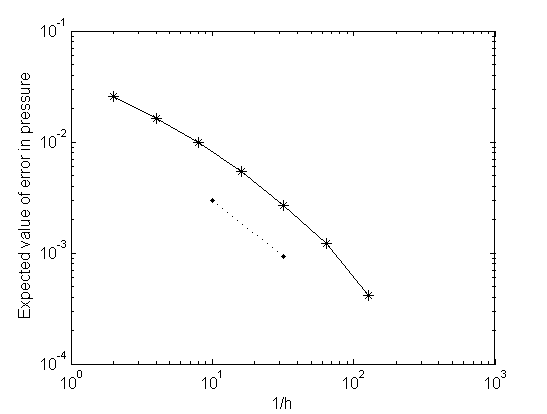} \ 
\includegraphics[width=0.5\textwidth]{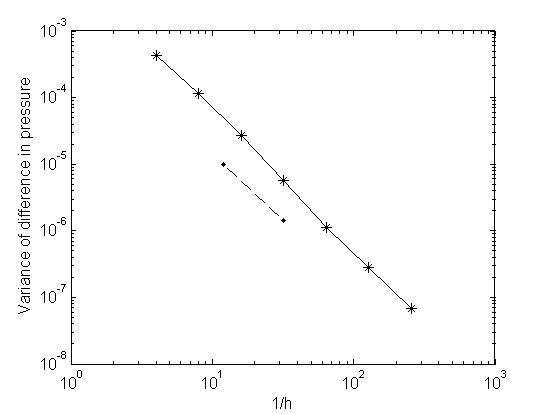}\hspace*{-0.75cm}
\caption{Left plot: $\left|\EE\left[M^{(2)}(u_{h^*})-M^{(2)}(u_{h})\right]\right|$, for 2D model
problem \eqref{modnum1} with $\lambda=0.3$, $\sigma^2=1$ and 
$h^*=1/256$. Right plot: Corresponding variance 
$\VV\left[M^{(2)}(u_{h})-M^{(2)}(u_{2h})\right]$. The gradient of the dotted 
(resp.~dashed) line is $-1$ (resp.~$-2$).}
\label{fig:midp}
\end{figure}

\begin{figure}[t!]
\centering
\hspace*{-0.75cm}\includegraphics[width=0.5\textwidth]{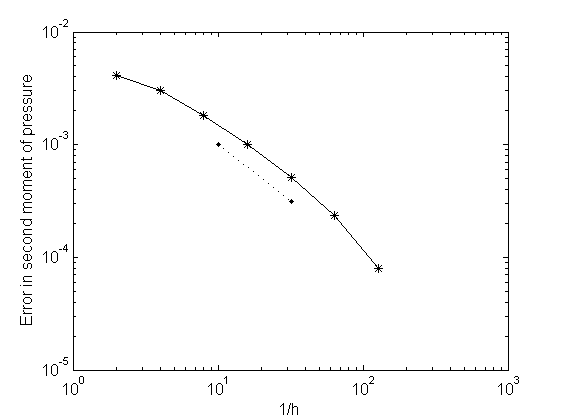} \ 
\includegraphics[width=0.5\textwidth]{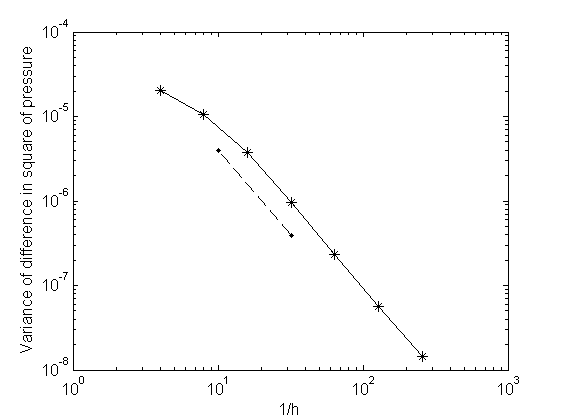}\hspace*{-0.75cm}
\caption{Left plot: $\left|\EE\left[M^{(2)}(u_{h^*})^2-M^{(2)}(u_{h})^2\right]\right|$ for 2D model problem \eqref{modnum1} with $\lambda=0.3$, $\sigma^2=1$ and 
$h^*=1/256$. Right plot: Corresponding variance 
$\VV\left[M^{(2)}(u_{h})^2-M^{(2)}(u_{2h})^2\right]$. The gradient of the dotted 
(resp.~dashed) line is $-1$ (resp.~$-2$).}
\label{fig:midp2}
\end{figure}

\begin{figure}[t]
\centering
\hspace*{-0.75cm}\includegraphics[width=0.5\textwidth]{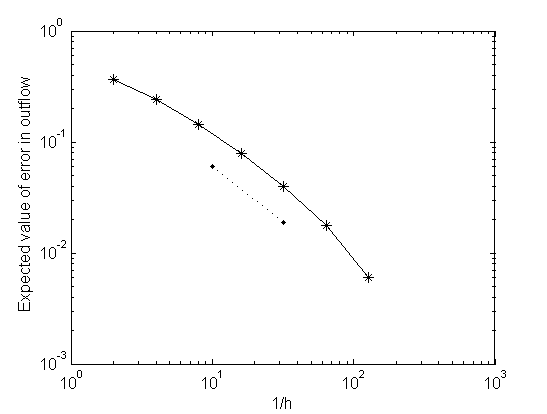} \ 
\includegraphics[width=0.5\textwidth]{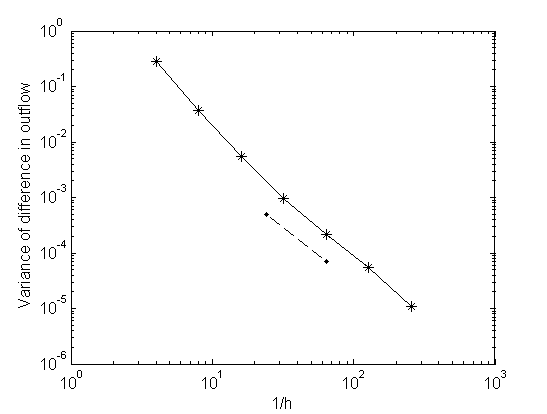}\hspace*{-0.75cm}
\caption{Left: Plot of $\big|\EE\big[M^{(4)}_\omega(u_{h^*})-M^{(4)}_\omega(u_{h})\big]\big|$, for 2D model \eqref{modnum2}
problem with $\lambda=0.3$, $\sigma^2=1$, $\psi=x_1$ and 
$h^*=1/256$. Right: Corresponding variance 
$\VV\big[M^{(4)}_\omega(u_{h})-M^{(4)}_\omega(u_{2h})\big]$. The gradient of the dotted 
(resp.~dashed) line is $-1$ (resp.~$-2$).}
\label{fig:out}
\end{figure}

For the second model problem \eqref{modnum2}, we consider an approximation of the average outflow through the boundary $\Gamma_\mathrm{out} := \{x_1=1\}$ computed via the functional $M^{(4)}_\omega$ in \S \ref{sec:funcex}. As the weight function we choose the linear function $\psi(x) = x_1$, which is equal to 1 at all nodes on $\Gamma_\mathrm{out}$ and equal to 0 at all other Dirichlet nodes. Thus,
$M^{(4)}_\omega(u)$ is exactly equal to the flow through $\Gamma_\mathrm{out}$.
As predicted we see again linear convergence in $h$ for $\big|\EE\big[M^{(4)}_\omega(u_{h^*})-M^{(4)}_\omega(u_{h})\big]\big|$, and quadratic convergence for $\VV\big[M^{(4)}_\omega(u_{h})-M^{(4)}_\omega(u_{2h})\big]$ in Figure~\ref{fig:out}.


\section{Level dependent estimators}
\label{sec:level}
The key ingredient in the multilevel Monte Carlo algorithm is the telescoping sum \eqref{eq:identity},
\begin{equation*}
\EE[Q_h] = \EE[Q_{h_0}] + \sum_{\ell=1}^L \EE[Q_{h_\ell} - Q_{h_{\ell-1}}].
\end{equation*}
Looking at this equation more carefully, we see that we are free to choose how to approximate $Q$ on the different levels, without violating the above identity, as long as the approximation of $Q_{h_\ell}$ is the same in the two terms in which it appears on the right hand side, for $\ell=0,...,L-1$. In particular, this implies that we do not have to approximate $Q$ on level $\ell -1$ in the same way as we approximate it on level $\ell$. We can, for example, approximate the coefficient $a(\omega,x)$ differently on each level, without introducing any additional bias in the final result $\EE[Q_h]$.

This is particularly useful in groundwater flow modelling, where the random fields $a(\omega,x)$ are highly oscillatory and vary on a fine scale. The coarsest grids of the (plain--vanilla) MLMC estimator will not be able to resolve the coefficient well. As a consequence of this, one needs to choose the coarsest grid size $h_0$ smaller than a certain threshold to get the MLMC estimator with the smallest absolute cost. This limits the number of levels and the amount of benefit that the MLMC estimator potentially offers. Numerical investigations in \cite{cgst11}, for example, show that for log-normal random fields $a(\omega,x)$ with exponential, 1-norm covariance function and correlation length $\lambda$, the optimal choice is $h_0 \approx \lambda$. A possible solution to this problem, which will allow us to choose $h_0$ independent of $\lambda$ and thus achieve higher gains, is to use smoother approximations of the coefficient on the coarser levels. We will present one way of doing this in \S \ref{sec:trunc}, where we use level-dependent truncations of the Karhunen-L\`oeve expansion of $a(\omega,x)$. 


Before we go on to analyse the level-dependent estimators for log--normal coefficient fields, we would like to point out that even though this strategy does not introduce any additional bias in the final result $\EE[Q_h]$, it may influence the values of the convergence rates $\alpha$ and $\beta$ in Theorem \ref{main_thm}. One has to be careful not to introduce any additional model/approximation errors that decay at a slower rate than the discretisation error.

\subsection{Truncated KL-expansions}\label{sec:trunc}

As an exemplary case, let us now consider log-normal random fields with exponential, 1-norm covariance, i.e. covariance function \eqref{cov:exp} with $\|x\|=\|x\|_1:= \sum_{i=1}^d |x_i|$. We will comment on the general case at the end of the section. 

For a Gaussian random field $g$, the Karhunen-L\`oeve (KL) expansion is an expansion in terms of 
a countable set of independent, standard Gaussian random variables 
$\{\xi_n\}_{n \in \mathbb{N}}$. It is given by
\begin{equation*}
g(\omega,x)= \EE\left[g(\omega,x)\right] + \sum_{n=1}^\infty \sqrt{\theta_n}b_n(x)\xi_n(\omega),
\end{equation*}
where $\{\theta_n\}_{n \in \mathbb{N}}$ are the eigenvalues and $\{b_n\}_{n \in \mathbb{N}}$ are
the corresponding normalised eigenfunctions of the covariance operator with kernel function 
\[
C(x,y) := \EE\Big[(g(\omega,x)-\EE[g(\omega,x)])(g(\omega,y)-\EE[g(\omega,y)]) \Big].
\] 
For more details on the derivation, see e.g. \cite{Ghanem_spanos}.

The log-normal coefficient field shall then be written as
\begin{equation*}
a(\omega,x)=\exp\left[ \EE\left[g(\omega,x)\right] + \sum_{n=1}^\infty \sqrt{\theta_n}b_n(x)\xi_n(\omega)\right],
\end{equation*}
and the random fields resulting from truncated expansions with $K \in \mathbb{N}$ terms shall be denoted by
\begin{equation*}
g_K(\omega,x) \, :=\EE\left[g(\omega,x)\right] + \sum_{n=1}^K \sqrt{\theta_n}b_n(x)\xi_n(\omega) \qquad \text{and} \qquad 
a_K(\omega,x)\, :=\exp\left[ g_K(\omega,x) \right].
\end{equation*}
Moreover, we denote by $u_K \in H^1_\phi(D)$ the weak solution to 
\begin{align}
\label{modtrunc}
-\nabla \cdot (a_K(\omega, x) \nabla u_K(\omega, x)) &= 
f(\omega,x), \ \qquad \mathrm{for} \ x \in D, \\
u_K(\omega,x) &= \phi_j(\omega,x), \qquad \mathrm{for} \ x \in \Gamma_j\,.  \nonumber
\end{align}
i.e. our model problem \eqref{mod} with the coefficient $a$ replaced by its $K$-term approximation. The finite element approximation of $u_K$ in $V_{h,\phi}$ is denoted by $u_{K,h}$. It has been shown in \cite{charrier,cst11} that in the case of the 1-norm exponential covariance, Assumptions A1--A2 are satisfied also for $a_K$, for any $t < 1/2$ (independent of $K$). Therefore the theory in the earlier sections applies also to \eqref{modtrunc}. 

Since the convergence with respect to $K$ is quite slow (see below), to get a good approximation to $\EE[Q_h]$ we need to include a large number of terms on the finest grid, both in the case of the standard and the MLMC estimator. However, as mentioned at the beginning of this section, we are free in the MLMC estimator to choose different approximations of $a(\omega,x)$ on the coarser levels. In particular, we can choose to include fewer terms in the KL-expansion above. The eigenvalues $\{\theta_n\}_{n \in \mathbb{N}}$ are all non--negative with $\sum_{n\geq 1}\theta_n<+\infty$. If we order them in decreasing order of magnitude, the corresponding eigenfunctions $\{b_n\}_{n \in \mathbb{N}}$ will be ordered in increasing order of oscillations over $D$. By truncating the KL-expansion after $K_\ell < K$ terms, we are hence disregarding the contributions of the most oscillatory eigenfunctions, and $a_{K_\ell}(\omega,x)$ is a smoother approximation of $a(\omega,x)$ than $a_{K}(\omega,x)$ leading to FE problems that can be solved more accurately on the coarser levels. 
The key question is then, how we should choose $K_{\ell}$ in terms of $\ell$ (or equivalently $h_\ell$). As an example of how to determine a suitable strategy, 
we make use of the following results from \cite{charrier,cst11} on the convergence of $u_{K,h}$ to $u$ in the 1-norm exponential covariance case. See below for comments on strategies for other fields.

\begin{proposition}
\label{thm:stochh1}
Let $a$ be a log--normal random field with 1-norm exponential covariance, and suppose that Assumption A3 is satisfied for some $p_* \in (0, \infty]$ and for $t \ge 1/2$. Then, 
$$
\|u- u_{K,h}\|_{L^p(\Omega,H^1_0(D))} \; \lesssim \; C_{a,f,\phi_j} \left(h^s + 
K^{-s}\right) \quad \text{and} \quad \|u- u_{K,h}\|_{L^p(\Omega,L^2(D))} \; \lesssim \; C_{a,f,\phi_j} \left(h^{2s} + K^{-s}\right),
$$
for all $p < p_*$ and $0<s<1/2$. The hidden constant is independent of $h$ and $K$.
\end{proposition}
As in the previous sections this result can again be extended in a straightforward way to functionals.

\begin{corollary}
\label{thm:stochfunc}
Let the assumptions of Proposition \ref{thm:stochh1} be satisfied and suppose that for the truncated problem \eqref{modtrunc} and for the functional $M_\omega(\cdot)$ we have Assumption F1 satisfied with $t_* \ge \frac{1}{2}$ and $q_* \in (0, \infty]$, i.e. $M_\omega$ is Fr\'echet differentiable and $\overline{D_v M_\omega}(u_K, u_{K,h})$ is bounded in $H^{1-t_*}(D)$. Assume further that there exists $C_F' \in L^{q_*}(\Omega)$ such that $\overline{D_v M_\omega}(u, u_{K}) \le C_F' \|v\|_{H^1(D)}$, for all $v \in H^1_0(D)$. Then
$$
\|M_\omega(u)- M_\omega(u_{K,h})\|_{L^p(\Omega)} \; \lesssim C_{a,f,\phi_j,C_F,C_F'} \left( h^{2s} + K^{-s}\right),
$$
for any $p < \left(\frac{1}{p_*} + \frac{1}{q_*} \right)^{-1}$ and $\;0<s<1/2$. The
hidden constant is again independent of $h$ and $K$.
\end{corollary}
\begin{proof}
First note that due to the triangle inequality, we have of course
\begin{align}
\label{ineq:triangle}
|M_\omega(u_{K,h})-M_\omega(u)| \;\leq \;|M_\omega(u_{K,h})-M_\omega(u_{K})| \;+ \;|M_\omega(u_K)-M_\omega(u)|
\end{align}
As noted above, it follows from \cite[\S 7]{charrier} that Assumptions A1--A2 are satisfied for the truncated expansion $a_K$ of a log--normal random field with 1-norm exponential covariance. Since Assumption~A3 is also assumed to hold, it follows as in Proposition \ref{funcconv} from H\"older's inequality that the $L^p$--norm of the first term in \eqref{ineq:triangle} is $\mathcal{O}(h^{2s})$, for any $p < \left(\frac{1}{p_*} + \frac{1}{q_*} \right)^{-1}$ and $\;0<s<1/2$, with a constant that is independent of $h$ and $K$.

To bound the second term in \eqref{ineq:triangle}, we can use \eqref{eq:mdiff} so that by assumption 
\begin{align}\label{eq:functrunc}
 |M_\omega(u)-M_\omega(u_{K})|
&= \overline{D_{u-u_K} M_\omega}(u, u_{K})
\leq C_F' \|u-u_{K}\|_{H^1(D)}\,.
\end{align}
It follows from \cite[Proposition 2.8]{cst11} that $\|u-u_{K}\|_{H^1(D)} \leq C_{a,f,\phi_j} K^{-s}$, for any $\;0<s<1/2$. Thus, 
H\"older's inequality implies again that the $L^p$--norm of the second term is $\mathcal{O}(K^{-s})$, for any $p < \left(\frac{1}{p_*} + \frac{1}{q_*} \right)^{-1}$ and $\;0<s<1/2$, with a constant that is independent of $h$ and $K$. 
Note that in \eqref{eq:functrunc} we cannot exploit Galerkin orthogonality to get a doubling of the convergence rate with respect to $K$, since $u$ and $u_K$ are solutions to two problems with different bilinear forms.
\end{proof}

As expected, these results suggest that to balance out the two error contributions, we should choose $K_\ell$ as a power of $h_\ell$. Note that a similar strategy was already suggested in the context of the related Brinkman problem in \cite{gkss11}. However, there, a certain decay rate for the FE error with respect to the number of KL-modes $K$ was assumed. Here we make no such assumption and instead use Proposition \ref{thm:stochh1}. For the simple functional $M(u):=|u|_{H^1(D)}$, Proposition \ref{thm:stochh1} implies $K_\ell \gtrsim h_\ell^{-1}$. For other functionals, that satisfy Assumption F1 with $t_* \ge t$, Corollary \ref{thm:stochfunc} implies that we should choose $K_\ell \gtrsim h_\ell^{-2}$. If we do this, we have the following results for the multilevel Monte Carlo convergence rates in Theorem \ref{main_thm}.

\begin{proposition}\label{mlmcconv_trunc}
Provided Assumption F1 is satisfied with $t_* \ge \frac{1}{2}$ and $K_\ell \gtrsim h_\ell^{-2}$, for all $\ell=0,\ldots,L$, then the convergence rate of the multilevel Monte Carlo method in \S\ref{sec:multilevel} does not deteriorate when approximating the functional $M_\omega(u_{h_\ell})$ by $Q_{h_\ell} := M_\omega(u_{K_\ell,h_\ell})$ on each level $\ell$. In particular, let the assumptions of Corollary \ref{thm:stochfunc} be satisfied with $p_* > 2$ and $q_* > \frac{2p_*}{p_*-2}$. Then the Assumptions M1--M2 in Theorem \ref{main_thm} hold for any $\alpha < 1$ and $\beta < 2$. If Assumption F1 is satisfied only for some $t_* < 1/2$, 
then $K_\ell \gtrsim h_\ell^{-(1+2t_*)}$ is a sufficient condition.
\end{proposition}
\begin{proof}
The proof is analogous to that of Proposition \ref{funcconv} using the result in Corollary \ref{thm:stochfunc}. The final statement follows from Remark \ref{rem:quad}.
\end{proof}

As before, in the presence of quadrature error (cf.~Remark \ref{rem:quad}), we will not be able to get $\mathcal{O}(h^{2s})$ convergence for the first term in \eqref{ineq:triangle} for the approximate finite element solution $\tilde u_{K,h}$. Due to the loss of Galerkin orthogonality for the primal problem, it is in general only possible to prove
$
|M_\omega(u)-M_\omega(\tilde u_{K,h})| = \mathcal{O} \left(h^s + K^{-s}\right).  
$
Thus with the quadrature error taken into account the optimal choice is $K_\ell \gtrsim h_\ell^{-1}$ for all functionals and we will always use that in our numerical tests in the next section. Higher rates of convergence can again be recovered, if the random field $a(\omega,x)$ is more regular.

Let us finish this section with some comments on truncated expansions $a_K=\exp(g_K)$ of log--normal fields with other covariance functions. The convergence rate of $|M_\omega(u)-M_\omega(u_{K})|$  depends in general on the rate of decay of the KL-eigenvalues $\theta_n$ and on the rate of growth of $\|\nabla b_n\|_\infty$. If we assume that $|M_\omega(u_K)-M_\omega(u_{K,h})| = \mathcal{O}(h^{s})$ and $|M_\omega(u)-M_\omega(u_{K})| = \mathcal{O}(K^{-\sigma})$, for some $0<s\le 1$ and $0 < \sigma < \infty$, then the number of KL-terms in a multilevel Monte Carlo method on each level should satisfy $K_{\ell} \gtrsim h_\ell^{-\frac{s}{\sigma}}$. For smoother fields (e.g. with covariance functions from the Mat\'ern class), $\frac{s}{\sigma}$ will usually be significantly smaller than $1$, and thus the number of KL-terms only needs to grow very slowly from level to level.  

However, the only other rigorous results regarding convergence rates for truncated expansions $a_K=\exp(g_K)$ of log--normal fields -- except those for the 1-norm exponential covariance above -- are for the case of a Gaussian covariance function
\begin{equation}\label{cov:gauss}
\mathbb E\Big[(g(\omega,x)-\mathbb E[g(\omega,x)])(g(\omega,y)-\mathbb 
E[g(\omega,y)]) \Big]= \sigma^2 \exp(-\|x-y\|^2/ \lambda^2)
\end{equation}
for $g$ with $\sigma^2$ and $\lambda$ as in \eqref{cov:exp}. In this case, provided the mean is sufficiently smooth, we in fact have $a(\omega, \cdot) \in \mathcal C^\infty(\overline D)$ and
$$
|M_\omega(u)-M_\omega(u_{K,h})| \; \lesssim C_{a,f,\phi_j} \left( h^2 \;+\; \exp\big(-c_1 \, K^{1/d}\big)\right),
$$
for some $c_1>0$ (cf.~\cite{cst11}), where $d$ is again the spatial dimension. Thus, $K_\ell$ only needs to be increased 
logarithmically with $h_\ell^{-d}$ in this case.

However, all these results are asymptotic results, as $h_\ell \to 0$, and thus they only guarantee that level-dependent truncations do not deteriorate the performance of the multilevel Monte Carlo method asymptotically as the tolerance $\varepsilon \to 0$. The real benefit of using level-dependent truncations is in absolute terms for a fixed tolerance $\varepsilon$, since the smoother fields potentially allow the use of coarser levels and thus significant gains in the absolute cost of the algorithm. In the next section, we see that this is in fact the case and we show the gains that are possible, especially for covariance functions with short correlation length $\lambda$.


\subsubsection{Numerics}\label{sec:numkl}

To be able to deal with very short correlation lengths in a reasonable time, we start with the 1D equivalent of model problem \eqref{modnum1}, on $D=(0,1)$. We take $a$ to be a log--normal random field with 1--norm exponential covariance function \eqref{cov:exp}, with correlation length $\lambda=0.01$ and variance $\sigma^2=1$. We will present results for two different modelling regimes for $a$: one in which the number of modes included is fixed at $K$, independent of $h_\ell$, and one in which the number of modes $K_\ell$ is chosen dependent on the mesh size $h_\ell$.  In order to make the two regimes comparable, we choose $K_\ell$ such that both regimes include the same number of modes (i.e. $K_\ell = K$) on the finest grid considered.

Figures \ref{fig:dropkl1} and \ref{fig:dropkl_1d_cost} show results for the point evaluation of the pressure at $x=2049/4096$, i.e.  $M^{(1)}(u)$ from \S \ref{sec:funcex} with $x^*=2049/4096$. Similar gains can be obtained for other quantities of interest.

\begin{figure}[t]
\centering
\hspace*{-0.75cm}\includegraphics[width=0.5\textwidth]{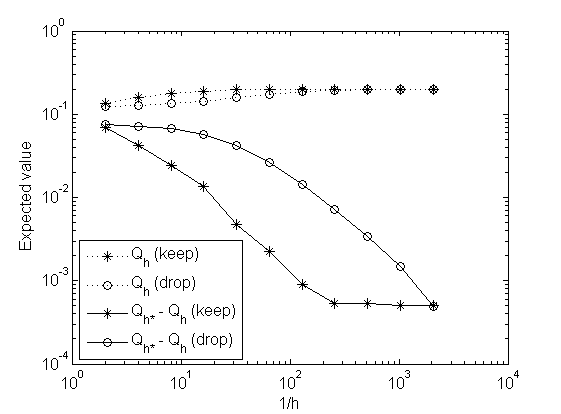} \ 
\includegraphics[width=0.5\textwidth]{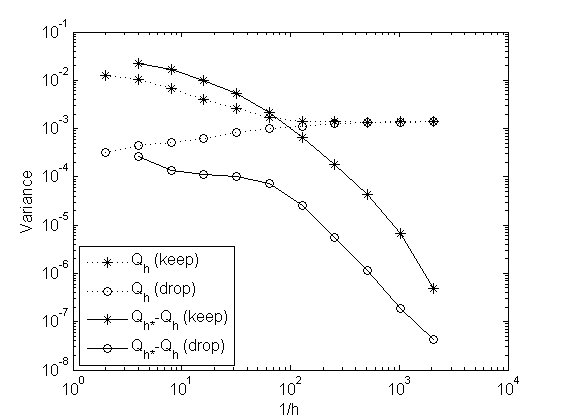}\hspace*{-0.75cm}
\caption{Left: Plot of $\EE\left[M^{(1)}(u_{h})\right]$ and $|\EE\left[M^{(1)}(u_{h^*})-|M^{(1)}(u_{h})\right]|$, for model problem \eqref{modnum1} with $d=1$, $\lambda=0.01$, $\sigma^2=1$, $K_\ell=h_\ell^{-1}$, $h^*=1/4096$, $K^*=4096$ and $x^*=2049/4096$. Right: Corresponding variances  $\VV\left[M^{(1)}(u_{h})\right]$ and $\VV\left[M^{(1)}(u_{h})-M^{(1)}(u_{2h})\right]$.}
\label{fig:dropkl1}
\end{figure}

\begin{figure}[t!]
\centering
\hspace*{-0.75cm}\includegraphics[width=0.5\textwidth]{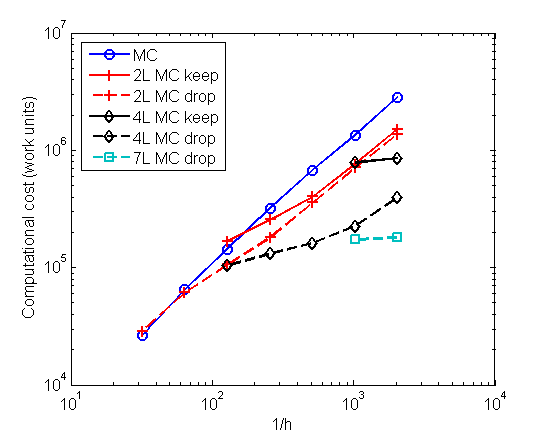} 
\caption{Plot of cost versus $1/h$ for a fixed tolerance of the sampling error of $\delta=10^{-3}$, for model problem \eqref{modnum1} with $d=1$, $\lambda=0.01$, $\sigma^2=1$ and $K_\ell = h_\ell^{-1}$. The quantity of interest is $M^{(1)}(u)$ with $x^*=2049/4096$.
}
\label{fig:dropkl_1d_cost}
\end{figure}

Let us start with Figure \ref{fig:dropkl1}. The number of modes included in the regime with a fixed number of modes is $K=2048$, and for the level--dependent regime we choose $K_\ell = h_\ell^{-1}$. The reference value $Q_{h*}$ is computed with $h^*=1/4096$ and $K^*=4096$. In the left plot, we see that even though dropping modes leads to a larger bias on the coarser grids, the fact that we chose $K_\ell=K$ on the finest grid ensures that the bias is the same on this grid.  It is the plot on the right that gives us information about the coarsest level we should include in the multilevel estimator. If we are in a situation where $\VV[Q_{h_\ell}-Q_{h_{\ell-1}}] \geq \VV[Q_{h_\ell}]$, then there is no benefit including level $\ell-1$ in the multilevel estimator, since it would only increase the cost of the estimator. Looking at the right plot in Figure \ref{fig:dropkl1}, it is then clear that for the regime with a fixed number of modes on each level, we should not include any levels coarser than $h_0 = 1/64 \, (\approx \lambda)$ in the estimator, as was already observed in \cite{cgst11}. With the level--dependent regime, however, it is viable to include levels as coarse as $h_0 = 1/2$. This leads to significant reductions in computational cost, as is shown in Figure \ref{fig:dropkl_1d_cost}. 

In Figure \ref{fig:dropkl_1d_cost}, we fix the required tolerance for the sampling error (i.e. the standard deviation of the estimator) at $\delta=10^{-3}$, and look at how the cost of the different estimators grows as we decrease the mesh size $h$ of the finest grid. The computational cost of the multilevel estimator is calculated as $N_0 h_0^{-1} + \sum_{\ell=1}^L N_\ell (h_\ell^{-1} + h_{\ell-1}^{-1})$ work units, since we know that $\gamma=1$ in (M3) for $d=1$. To make the estimators comparable, on each grid $h_\ell$, the standard Monte Carlo estimator is computed with $K_\ell$ modes, the "MLMC keep" estimator is computed with $K=K_\ell$ modes on all levels, and the "MLMC drop" estimator is computed with a varying number $K_\ell=h_\ell^{-1}$ modes on the levels. We clearly see the benefit of using the level--dependent multilevel estimator. For example, on the grid of size $h=1/2048$, the cheapest multilevel estimator with a fixed number of modes is the 4 level estimator, which has a cost of $8.6 \times 10^5$ work units. The cheapest level--dependent multilevel estimator, on the other hand, is the 7 level estimator, whose computational cost is only $1.8 \times 10^5$ units. For comparison, the cost of the standard estimator on this grid is $2.8 \times 10^6$ units. 

An important point we would like to make here, is that not only do the level--dependent estimators have a smaller absolute cost than the estimators with a fixed number of modes, they are also a lot more robust with respect to the coarse grids included. On the $h=1/2048$ grid, the 11 level estimator (i.e. $h_0=1/2$) with fixed $K$, costs $1.1 \times 10^{7}$ units, which is 4 times the cost of the standard MC estimator. The 11 level estimator with level--dependent $K_\ell$ costs $2.4 \times 10^5$ units, which is only marginally more than the best level--dependent estimator (the 7 level estimator).

For practical purposes, the real advantage of the level--dependent approach is evident on coarser grids. We see in Figure \ref{fig:dropkl_1d_cost} that on grids coarser than $h=1/256$, all multilevel estimators with a fixed number of modes are more expensive than the standard MC estimator. With the level--dependent multilevel estimators on the other hand, we can make use of (and benefit from) multilevel estimators on grids as coarse as $h=1/64$. This is very important, especially in the limit as the correlation length $\lambda \rightarrow 0$, as eventually all computationally feasible grids will be "coarse" with respect to $\lambda$. With the level--dependent estimators, we can benefit from the multilevel approach even for very small values of $\lambda$.

Let us now move on to a model problem in 2D. We will study the flow cell model problem \eqref{modnum2} on $D=(0,1)^2$, and take the outflow functional $M^{(4)}_\omega(u)$ from \S \ref{sec:funcex} as our quantity of interest. As in \S \ref{sec:num_funcs}, we choose the weight function $\psi = x_1$. We choose $a$ to be a log--normal random field with 1--norm exponential covariance function \eqref{cov:exp}, with $\lambda=0.1$ and $\sigma^2=1$. 

Figure \ref{fig:dropkl_2d} is similar to Figure \ref{fig:dropkl1}. The number of modes included in the regime with a fixed number of modes is $K=512$, and in the level--dependent regime we include $K_\ell=4h_\ell^{-1}$ modes on each level. The reference value $Q_{h*}$ is computed with $h^*=1/256$ and $K^*=1024$. As before, the coarsest level which should be included in the multilevel estimator can be estimated from the right plot in Figure \ref{fig:dropkl_2d}. For the regime with a fixed number of modes, it is clear that no grids coarser than $h_0=1/8$ should be included in the multilevel estimator. For the level--dependent regime, it is viable to include grids as coarse as $h_0=1/2$.

In Figure \ref{fig:dropkl_2d_cpu}, we see the gains in computational cost that are possible with the level--dependent estimators. The results shown are calculated with a Matlab implementation on a 3GHz Intel Core 2 Duo E8400 processor with 3.2GByte of
RAM, using the sparse direct solver provided in Matlab through the standard backslash operation
to solve the linear systems for each sample. Since we do not know the value of $\gamma$ in (M3) theoretically, we quantify the cost of the estimators by the CPU--time.  On the finest grid $h=1/256$, we clearly see a benefit from the level--dependent estimators. The cheapest multilevel estimator with a fixed number of modes is the 5 level estimator, with takes 13.5 minutes. The cheapest level--dependent estimator is the 7 level estimator, which takes only 2.5 minutes. For comparison, the standard MC estimator takes more than 7.5 hours.

\begin{figure}[t!]
\centering
\hspace*{-0.75cm}\includegraphics[width=0.5\textwidth]{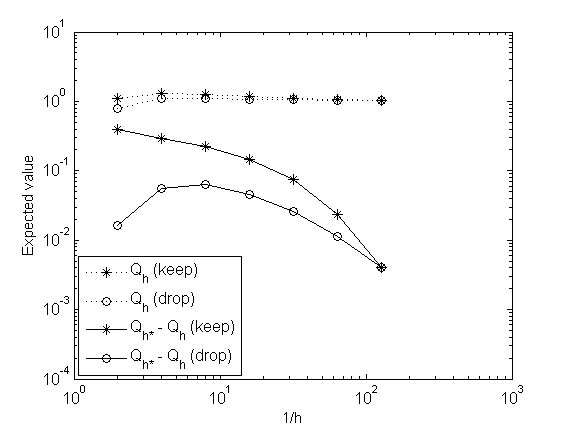} \ 
\includegraphics[width=0.5\textwidth]{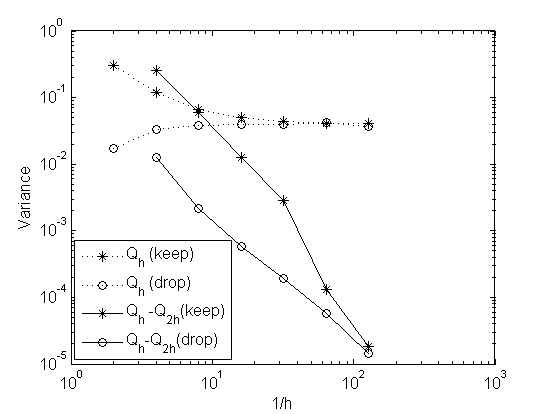}\hspace*{-0.75cm}
\caption{Left: Plot of $\big|\EE\big[M^{(4)}_\omega(u_{h^*})\big]\big|$ and $\big|\EE\big[M^{(4)}_\omega(u_{h^*})-M^{(4)}_\omega(u_{h})\big]\big|$, for model \eqref{modnum2}
problem with $d=2$, $\lambda=0.1$, $\sigma^2=1$, $\psi=x_1$, $K_\ell=4h_\ell^{-1}$, $h^*=1/256$ and $K^*=1024$. Right: Corresponding variances  $\VV\big[M^{(4)}_\omega(u_{h})\big]$ and $\VV\big[M^{(4)}_\omega(u_{h})-M^{(4)}_\omega(u_{2h})\big]$.}
\label{fig:dropkl_2d}
\end{figure}

\begin{figure}[h!]
\centering
\hspace*{-0.75cm}\includegraphics[width=0.5\textwidth]{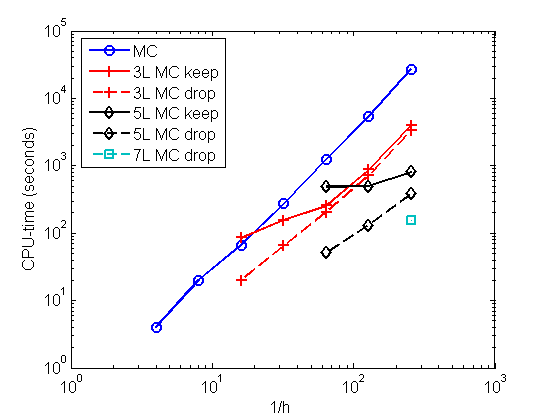} 
\caption{Plot of CPU-time versus $1/h$ for a fixed tolerance of the sampling error of $\delta=10^{-3}$, for model problem \eqref{modnum2} with $d=2$, $\lambda=0.1$, $\sigma^2=1$ and $K_\ell = 4h_\ell^{-1}$. The quantity of interest is $M^{(4)}_\omega(u)$, with $\psi =x_1$.
}
\label{fig:dropkl_2d_cpu}
\end{figure}


\section{Domains with corners and discontinuous coefficients}\label{sec:reg}

We now come to the last and most technical part of the paper. The first aim is to prove Theorem~\ref{regu}, i.e. to extend the regularity results in \cite{cst11} to piecewise $\mathcal{C}^2$ domains. In this situation, the solution $u$ can have singularities near the non--smooth parts of the boundary $\Gamma$, i.e. near corners in 2D and near corners and edges in 3D. These singularities can reduce the overall regularity of $u$, and hence need to be analysed. However, we will see in \S \ref{sec:poly} that under Assumptions A1-A2, this question can be reduced to analysing the singularities of the Laplace operator on $D$. We will follow \cite[\S5.2]{grisvard2}, and as in \cite{cst11} we will again establish the result first ``pointwise'' almost surely in $\omega \in \Omega$. The key technicality will again be to track how the constants in all the necessary estimates, in particular in the semi-Fredholm property of the underlying random differential operator, depend on $\omega$. 

In \S \ref{sec:dis}, we then extend the results also to the practically very important case where the coefficient $a(\omega,x)$ is discontinuous. This is of interest for example in subsurface flow modelling, where one often deals with layered media. If $a(\omega, \cdot)$ is piecewise H\"older continuous, then the regularity results from Theorem \ref{regu} and \S \ref{sec:poly} will still hold on each of the subdomains, but no longer globally on the entire domain. The aim of \S \ref{sec:dis} is to formulate an assumption similar to Assumption A2, under which we can conclude on the global regularity of $u$ also in the case of discontinuous coefficients.

\subsection{Regularity of random differential operators in domains with corners}
\label{sec:poly}

Let us recall that $D$ was assumed to be a bounded, Lipschitz polygonal/polyhedral domain in $\mathbb{R}^d$, $d = 2, 3$, and that $\lambda_\Delta(D) \in (0,1]$ is the largest number such that for all $0 < s \leq \lambda_\Delta(D), \, s \neq \frac{1}{2}$, the Laplace operator with homogeneous Dirichlet boundary conditions is surjective as an operator from $H^{1+s}(D) \cap H^1_0(D)$ to $H^{s-1}(D)$ (cf. Definition \ref{def:laplace}). As in \cite[\S5.2]{grisvard2}, for simplicity we actually consider $D$ to be a piecewise $\mathcal C^2$ domain and restrict ourselves for the most part to $\mathbb{R}^2$. However, we will also comment on the case $d=3$ in Remark \ref{remark:3D}(c) below. We again write the boundary $\Gamma$ as $\Gamma = \cup_{j=1}^m \Gamma_j$, where now in 2D each $\Gamma_j$ is an open arc of curve of class $\mathcal C^2$, and $\overline \Gamma_j$ meets $\overline \Gamma_{j+1}$ at $S_j$ (where we identify $\Gamma_{m+1}$ and $\Gamma_1$). 
We consider only domains with boundaries that are rectilinear near the corners, which of course includes Lipschitz polygonal/polyhedral domains. This means that at each corner $S_j$, we can find a polygonal domain $W_j \subset D$ such that the boundary $\partial W_j$ coincides with $\Gamma$ near $S_j$. 

Applying the Lax-Milgram Theorem, a unique variational solution $u(\omega, \cdot) \in H^1_0(D)$ to our model problem \eqref{mod} in the curvilinear polygon $D$ exists, for almost all $\omega \in \Omega$ (i.e. for all $\omega \in \Omega$ with $a_\mathrm{min}(\omega) > 0$ and $a_\mathrm{max}(\omega) < \infty$). Using Assumptions A1--A3, we can conclude as in \cite{cst11} that $u \in L^p(\Omega, H^1_0(D))$, for all $p< p_*$. The fact that $D$ is no longer $\mathcal C^2$ is of no relevance here. To prove more spatial regularity on $u$, we will now follow the proof in \S 5.2 of \cite{grisvard2}.

For a given $\omega \in \Omega$, with $a_\mathrm{min}(\omega) > 0$ and $a_\mathrm{max}(\omega) < \infty$, we define the differential operator
\[
A_\omega u = -\nabla \cdot (a(\omega, \cdot) \nabla u)).
\]
The following key result, which is based on \cite[Theorem 5.26]{kato}, is proved via a homotopy method in the proof of \cite[Lemma 5.2.5]{grisvard2}, for $s=1$. The proof for $s < 1$ is analogous.
\begin{lemma}
\label{lem:grisvard}
Let $m=1$ and $\omega \in \Omega$. If $0 < s \le \lambda_\Delta(D)$ and if there exists $C_{\scriptscriptstyle \mathrm{semi}}(\omega) > 0$ such that
\begin{equation}\label{semifred}
\|v\|_{H^{1+s}(D)} \leq C_{\scriptscriptstyle \mathrm{semi}}(\omega) \|A_\omega v \|_{H^{s-1}(D)},
\quad \text{for all} \ v \in H^{1+s}(D) \cap H^1_0(D),
\end{equation}
then $A_\omega$ is surjective from $H^{1+s}(D) \cap H^1_0(D)$ to $H^{s-1}(D)$.
\end{lemma}
Thus, if we can establish \eqref{semifred}, which essentially means that 
$A_\omega$ is semi-Fredholm as an operator from $H^{1+s}(D) \cap H^1_0(D)$ to 
$H^{s-1}(D)$, for some $s \le \lambda_\Delta(D)$,
then we can also conclude on the 
regularity of solutions of the stochastic variational problem \eqref{mod}. 
The following lemma essentially follows \cite[Lemma 5.2.3]{grisvard2}. 
However, in the case of a random coefficient, we crucially need to make sure 
that the constant $C_{\scriptscriptstyle \mathrm{semi}}(\omega)$ in \eqref{semifred} 
has sufficiently many moments as a random field on $\Omega$. To ensure this 
we need to carefully track the dependence on $a$ in the bounds in 
\cite[Lemma 5.2.5]{grisvard2}.
\begin{lemma}\label{regpol} 
Let $m \in \mathbb{N}$ and let Assumptions A1 --A2 hold for
some $0< t \le 1$. Then \eqref{semifred} holds for all $0 < s < t$ and 
$s \le \lambda_\Delta(D)$, $s \neq \frac{1}{2}$, with  
\begin{equation}\label{bound1}
C_{\scriptscriptstyle \mathrm{semi}}(\omega) \;:=\; \frac{a_{\mathrm{max}}(\omega)\|a(\omega, \cdot)\|^2_{\mathcal C^{t}(\overline D)}}{a_{\mathrm{min}}(\omega)^4}.
\end{equation}
In the case $t= \lambda_\Delta(D)=1$, \eqref{semifred} also holds for $s=1$, i.e. for the $H^2(D)$-norm.
\end{lemma}
\begin{proof}
We first consider the case where $m=1$ and $t = \lambda_\Delta(D)=1$. Note that the case $m=1$ is all that is needed to prove Lemma \ref{lem:grisvard}. We prove the more general case $m \in \mathbb{N}$ so that we can apply the bound \eqref{semifred} to any polygonal domain in the proof of Theorem \ref{regu}. For ease of notation, we suppress the dependence on $\omega$ in the coefficient, and denote $A_\omega$ simply by $A$ and $a(\omega,x)$ by $a(x)$.

We will prove \eqref{semifred} by combining the regularity results of $A$ in $\mathcal C^2$ domains, with regularity results of the Laplace operator $-\Delta$ on polygonal domains. Since we assume that $\Gamma$ is rectilinear near $S_1$, we can find a polygonal domain $W$ such that $W \subset D$ and $\partial W$ coincides with $\Gamma$ near $S_1$. Let $v \in H^{2}(D) \cap H^1_0(D)$ and let $\eta$ be a smooth cut-off function with support in $W$, such that $\eta \equiv 1$ near $S_1$ and then consider $\eta v$ and $(1-\eta)v$ separately. We start with $\eta v$.

Let $w \in H^2(W) \cap H^1_0(W)$. Since $\lambda_\Delta(D)=1$, we have for any polygonal domain $W$ the estimate
\begin{equation}\label{eq:lapint}
\|w\|_{H^2(W)} \,\lesssim\, \|\Delta w \|_{L^2(W)}\,,
\end{equation}
where the hidden constant depends only on $W$ (cf.~\cite{grisvard2}). Hence,
\begin{align*}
a(S_1) \,\|w\|_{H^2(W)} 
&\lesssim   \|A w \|_{L^2(W)} +  \|A w - a(S_1) \,\Delta w \|_{L^2(W)} \\
&\lesssim  \|A w \|_{L^2(W)} +  |(a(\cdot)-a(S_1) ) \nabla w |_{H^1(W)} \,.
\end{align*}
Now, using \cite[Lemma A.2]{cst11} (see also Theorem 6.2.25 in \cite{hackbusch}) we get 
\begin{align} \label{eq:pert}
a(S_1) \,\|w\|_{H^2(W)} &\;\lesssim  \;\|A w \|_{L^2(W)} \,+ \, |a|_{\mathcal C^1(\overline W)}|w|_{H^1(D)} \,+\, \|a-a(S_1)\|_{\mathcal C^0(\overline W)}\|\nabla w \|_{H^1(W)} \, . 
\end{align}
Using integration by parts and the fact that $w=0$ on $\partial W$, we have 
\begin{align*}
a_\mathrm{min}\,|w|^2_{H^1(W)} &\;\leq\; \int_W a |\nabla w|^2 \dx 
\;=\;  \int_W w \nabla \cdot (a \nabla w) \dx 
\end{align*}
and so via the Cauchy-Schwarz and the Poincar\'e inequalities
\begin{equation}\label{eq:intpar}
|w|_{H^1(W)} \lesssim  \frac{1}{a_\mathrm{min}} \|A w \|_{L^2(W)}.
\end{equation}

Denote now by $C$ the best constant such that \eqref{eq:pert} holds.
Since $a$ was assumed to be in $C^1(\overline W)$, we can choose $W$ (and hence the support of $\eta$) small enough so that
\begin{equation}\label{eq:suppeta}
C \|a-a(S_1)\|_{\mathcal C^0(\overline W)} \leq \frac{1}{2} \, a(S_1)
\end{equation}
Then, substituting \eqref{eq:intpar} and \eqref{eq:suppeta} into \eqref{eq:pert} and using $\|\nabla w \|_{H^1(W)} \leq \|w\|_{H^2(W)}$ and 
$a_\mathrm{min} \leq a_\mathrm{max}$ we have 
\begin{align}\label{eq:pol}
a(S_1) \,\|w\|_{H^2(W)} &\leq 2 C  \left( 1 + \frac{|a|_{\mathcal C^1(\overline W)}}{a_\mathrm{min}} \right) \|A w \|_{L^2(W)} \, \lesssim \, \frac{ \|a\|_{\mathcal C^1(\overline W)}}{a_\mathrm{min}} \|A w \|_{L^2(W)}\,. 
\end{align}

Since $v \in H^2(D) \cap H^1_0(D)$ and $W$ contains the support of $\eta$, we have $\eta v \in H^2(W) \cap H^1_0(W)$ and so estimate \eqref{eq:pol} applies to $\eta v$. Thus
\[
\|\eta v\|_{H^2(D)} \lesssim  \frac{ \|a\|_{\mathcal C^1(\overline W)}}{a_\mathrm{min}^2} \|A (\eta v) \|_{L^2(W)}.
\]

Let us move on to $(1-\eta)v$. Let $D' \subset D$ be a $\mathcal C^2$ domain that coincides with $D$ outside of the region where $\eta = 1$. This is always possible due to our assumptions on the geometry of $D$ near $S_1$. Then $(1-\eta) v \in H^2(D') \cap H^1_0(D')$, and using \cite[Proposition 3.1]{cst11} we have 
\[
\|(1-\eta) v\|_{H^{2}(D)} \lesssim \frac{a_{\mathrm{max}}\|a\|_{\mathcal C^{1}(\overline D')}}{a_{\mathrm{min}}^3} \, \|A \left((1-\eta) v\right)\|_{L^{2}(D')}.
\]
Adding the last two estimates together and using the triangle inequality, we have 
\begin{align}\label{eq:estv}
\|v\|_{H^2(D)} 
&\;\lesssim\;  \frac{ \|a\|_{\mathcal C^1(\overline D)}}{ a_{\mathrm{min}}^2} \left(\|A (\eta v) \|_{L^2(W)} \,+\,\frac{a_{\mathrm{max}}}{a_{\mathrm{min}}} \, \|A ((1-\eta) v)\|_{L^{2}(D')}\right).
\end{align}

It remains to bound the term in the bracket on the right hand side of \eqref{eq:estv} in terms of $\|A v \|_{L^2(D)}$. Note that
\[
A(\eta v) = \eta (A v) + 2 a \nabla \eta \cdot \nabla v + (A \eta) v.
\]
Thus, applying the triangle inequality and using the fact that $\eta$ was assumed to be smooth with $0 \leq \eta \leq 1$, we get
\begin{equation}
\label{ineq:eta_v}
\|A (\eta v) \|_{L^2(W)} \lesssim \|A v \|_{L^2(W)} + a_\mathrm{max} |v|_{H^1(W)} + 
\|a\|_{\mathcal C^{1}(\overline D)} \| v \|_{L^2(W)}.
\end{equation} 
The hidden constant depends on $\|\nabla \eta\|_{L^\infty(W)}$ and on $\|\Delta \eta\|_{L^\infty(W)}$. Finally using Poincar\'e's inequality on all of $D$, as well as an elliptic estimate similar to \eqref{eq:intpar} for $v$, i.e. $|v|_{H^1(D)} \leq \|A v \|_{L^2(D)} / a_{\mathrm{min}}$, leads to
\[
\|A (\eta v) \|_{L^2(W)} \lesssim \frac{\|a\|_{\mathcal C^{1}(\overline D)}}{a_{\mathrm{min}}} \|A v \|_{L^2(D)}.
\]
Substituting this and the corresponding bound for $\|A ((1- \eta) v) \|_{L^2(D')}$ into \eqref{eq:estv}, we finally get 
\begin{align*}
\|v\|_{H^2(D)} &\;\lesssim\; \frac{a_{\mathrm{max}}\|a\|^2_{\mathcal C^{1}(\overline D)}}{a_{\mathrm{min}}^4}  \, \|A v\|_{L^{2}(D)}
\end{align*}
for all $v \in H^2(D) \cap H^1_0(D)$. 
This completes the proof for the case 
$m=1$ and $t=\lambda_\Delta=1$.

The proof for $t<1$ and/or $\lambda_\Delta(D) < 1$ follows exactly the same lines. Instead of \eqref{eq:lapint}, we start with the estimate
\begin{equation}\label{eq:lapfrac}
\|w\|_{H^{1+s}(W)} \lesssim \|\Delta w \|_{H^{s-1}(W)},
\end{equation}
which holds for any $0< s \le \lambda_\Delta(D)$, $s \neq \frac{1}{2}$, and the hidden constant depends again only on $W$ (cf. \cite{bdln92} for example). Using \cite[Lemma A.1]{cst11} (see also Theorem 9.1.12 in \cite{hackbusch}), one can derive the following equivalent of \eqref{eq:pert}, for any $s \neq \frac{1}{2}$:
\begin{align*} \label{eq:pert_s}
a(S_1) \,\|w\|_{H^{1+s}(W)} &\;\lesssim\; \|A w \|_{H^{s-1}(W)} +  |a|_{\mathcal C^t(\overline W)}| w |_{H^1(W)} + \|a-a(S_1)\|_{\mathcal C^0(\overline W)}\|\nabla w \|_{H^s(W)}\,. 
\end{align*}
As before, the $|w|_{H^1(W)}$ term can be bounded using integration by parts, H\"older's inequality and the Poincar\'e inequality:
\begin{align*}
a_\mathrm{min} \, | w |^2_{H^1(W)} &\leq \|w\|_{H^{1-s}(W)} \|A w \|_{H^{s-1}(W)}
\lesssim \|w\|_{H^{1}(W)} \|A w \|_{H^{s-1}(W)} 
\lesssim |w|_{H^1(W)} \|A w \|_{H^{s-1}(W)}.
\end{align*}
The remainder of the proof requires only minor modifications. 

The case $m>1$ is treated by repeating the above procedure with a different cut--off function $\eta_j$ at each corner $S_j$. Estimate \eqref{eq:pol} applies to $\eta_j v$, for all $j=1,\dots,m$, and the regularity estimate for $\mathcal C^2$ domains from \cite{cst11} applies to $(1-\sum_{j=1}^n \eta_j) v$.
%
\end{proof}

\begin{remark}\em Lemma \ref{regpol} excludes the case $s=\frac{1}{2}$. However, an inequality very similar to \eqref{semifred} can easily be proved also in this case. Since $\|v\|_{H^{1+s}(D)} \leq \|v\|_{H^{1+t}(D)}$, for any $s \leq t$, $\|v\|_{H^{3/2}(D)}$ can also be bounded, as in \eqref{semifred}, if the $H^{1/2}(D)$--norm on the right hand side is replaced by the $H^{1/2+\delta}(D)$--norm, for some $\delta > 0$. 
\end{remark}

We are now ready to prove Theorem \ref{regu} for $d=2$. For the case $d=3$, see Remark \ref{remark:3D}(c).


\begin{proof}[Proof of Theorem \ref{regu}]
Let $d=2$ and suppose $u = u(\omega,\cdot)$ is the unique solution of \eqref{mod}. Let us first consider the case $\phi \equiv 0$. In this case, the fact that $u \in H^{1+s}(D) \cap H^1_0(D)$ and the bound on $\|u\|_{H^{1+s}(D)}$ in \eqref{ineq:regu} follow immediately from Lemmas \ref{lem:grisvard} and \ref{regpol}, for any $s < t$ and $s \le \lambda_\Delta(D)$, as well as for $s=1$ if $t=\lambda_\Delta(D)=1$, since $f = Au$. 

The case $\phi \neq 0$ now follows from a simple trace theorem, see e.g. \S 1.4 in \cite{grisvard}. We will only show the proof for $t=\lambda_\Delta(D)=1$ in detail. Due to Assumption A3 we can choose $\phi \in H^2(D)$ with $\|\phi\|_{H^2(D)} \lesssim \sum_{j=1}^m \|\phi_j\|_{H^{3/2}(\Gamma_j)}$, and so $f_0 := f - A\phi \in L_2(D)$. Since $u_0 := u - \phi \in H^1_0(D)$ we can apply the result we just proved for the case $\phi \equiv 0$ to the problem $A u_0 = f_0$ to get 
\begin{align*}
\|u_0\|_{H^2(D)} &\;\lesssim\; C_{\scriptscriptstyle \mathrm{semi}}(\omega) \left(\|A u_0 \|_{L^2(D)} + \|A \phi \|_{L^2(D)}\right)\\
&\;\lesssim\; C_{\scriptscriptstyle \mathrm{semi}}(\omega) \left(\|f\|_{L^2(D)} + \|a\|_{C^1(\overline D)} \, \|\phi\|_{H^2(D)} \right),
\end{align*}
where in the last step we have again used \cite[Lemma A.2]{cst11}. The claim of the Theorem for $\phi \not\equiv 0$ then follows by the triangle inequality.
\end{proof}

\begin{remark} \label{remark:3D}
\em \renewcommand{\labelenumi}{(\alph{enumi})}
\begin{enumerate}
\item The behaviour of the Laplace operator near corners is described in detail in \cite{grisvard2,grisvard}. In particular, in the pure Dirichlet case for convex domains we always get $\lambda_\Delta(D) = 1$. For non-convex domains $\lambda_\Delta(D) = \min_{j=1}^m \pi/\theta_j$, where $\theta_j$ is the angle at corner $S_j$. Hence, $\lambda_\Delta(D) > 1/2$ for any Lipschitz polygonal domain.
\item In a similar manner one can prove regularity of $u$ also in the case of Neumann and mixed Dirichlet/Neumann boundary conditions provided the boundary conditions are compatible, like in our model problem \eqref{modnum2}. For example, in order to apply the same proof technique used here at a point where a Dirichlet and a homogeneous Neumann boundary meet, we can first reflect the problem and the solution across the Neumann boundary. Then we apply the above theory on the union of the original and the reflected domain. The regularity for the Laplacian is in general lower in the mixed Dirichlet/Neumann case than in the pure Dirichlet case. In particular, $\lambda_\Delta(D) = \min_{j=1}^m \frac{\pi}{2\theta_j}$ in the mixed case in 2D and so full regularity (i.e. $\lambda_\Delta(D) = 1$) is only possible, if all angles are less than $\pi/2$. For an arbitrary Lipschitz polygonal domain we can only guarantee $\lambda_\Delta(D) > 1/4$.
\item The 3D case is similar, but in addition to singularities at corners (for which the analysis is identical to the above) we also need to consider edge singularities. This is a bit more involved and we refer to \cite[\S8.2.1]{grisvard2} for more details. However, provided $D$ is convex, we obtain again $\lambda_\Delta(D) = 1$ always in the pure Dirichlet case.
\end{enumerate}
\end{remark}

\subsection{Discontinuous coefficients}\label{sec:dis}

We now shift our attention from the domain $D$ to the random coefficient $a(\omega,x)$. 
In practice, one is often interested in models with discontinuous coefficients, e.g. modelling different rock strata in the subsurface. Such coefficients do not satisfy Assumption A2, and the regularity results from Theorem \ref{regu} can not be applied directly. However, this loss of regularity is confined to the interface between different strata and it is still possible to prove a limited amount of regularity even globally. 

Let us consider \eqref{mod} on a Lipschitz polygonal domain $D \subset \mathbb{R}^2$ that can be decomposed into disjoint Lipschitz polygonal subdomains $D_k$, $k=1,\ldots,K$. Let $PC^{t}(\overline D) \subset L_\infty(D)$ denote the space of piecewise $C^{t}$ functions with respect to the partition $\{D_k\}_{k=1}^K$ (up to the boundary of each region $D_k$). We replace Assumption A2 by the following milder assumption on the coefficient function $a$:
\begin{itemize}
\item[{\bf A2*.}\!\!] \ $a \in L^{p}(\Omega,PC^{t}(\overline D))$, for some $0 < t \le 1$ and for all $p \in (0,\infty)$.
\end{itemize}

Our regularity results for discontinuous coefficients rely on the following result from \cite{grisvard2,petzoldt_thesis}. The proof of this result uses the fact that for $0 \leq s<1/2$, $w \in H^s(D_i)$ if and only if the extension $\tilde w$ of $w$ by zero is in $H^s({\mathbb R^d})$.
\begin{lemma}\label{lem_petz} Let $v \in H^1(D)$ and $s< 1/2$, and suppose that $v \in H^{1+s}(D_k)$, for all $k=1,\ldots,K$. Then $v \in H^{1+s}(D)$ and
\[
\|v\|_{H^{1+s}(D)} = \|v\|_{H^{1}(D)}  +  \sum_{k=1}^K |v|_{H^{1+s}(D_k)} \,.
\]
\end{lemma}
 Thus, we cannot expect more than $H^{3/2-\delta}(D)$ regularity globally in the discontinuous case. However, as in the case of continuous fields, the regularity of the solution will also depend on the parameter $t$ in Assumptions A2* and A3 (i.e. on the H\"older/Sobolev regularity of $a$ and $f$, respectively), as well as on the behaviour of $A_\omega$ at any singular points. Since Lemma \ref{lem_petz} restricts us to $s < 1/2$ and since $\lambda_\Delta(D) > 1/2$ for any Lipschitz polygonal $D$ in the case of a pure Dirichlet problem, we do not have to worry about corners. Instead we define the set of {\em singular} (or {\em cross}) {\em points} $\mathcal{S}^\times := \{S^\times_\ell:\ell =1,\ldots,L\}$ to consist of all points $S^\times_\ell$ in $D$ where three or more subdomains meet, as well as all those points $S^\times_\ell$ on $\partial D$ where two or more subdomains meet. By the same arguments as in \S\ref{sec:poly}, the behaviour of $A_\omega$ at these singular points is again fully described by studying transmission problems for the Laplace operator, i.e. elliptic problems with piecewise constant coefficients, locally near each singular point (cf. \cite{nicaise,costabel,petzoldt_thesis}).
\begin{definition}
\em  Denote by $T(\alpha_1,\ldots,\alpha_K)$ the operator corresponding to the transmission problem for the Laplace operator with (constant) material parameter $\alpha_k$ on subdomain $D_k$, $k=1,\ldots,K$. Let $0 \le \lambda_{T}(D) \le 1$ be such that $T(\alpha_1,\ldots,\alpha_K)$ is a surjective operator from $H^{ 1+s}(D) \cap H^1_0(D)$ to $H^{s-1}(D)$, for any choice of $\alpha_1, \ldots, \alpha_K$ and for $s \le \lambda_{T}(D)$, $s\not=1/2$. In other words, $\lambda_{T}(D)$ is a bound on the order of the strongest singularity of $T(\alpha_1,\ldots,\alpha_K)$. 
\end{definition}
Without any assumptions on the partition $\{D_k\}_{k=1}^K$ or any bounds on the constants $\{\alpha_k\}_{k=1}^K$ it is in general not possible to choose $\lambda_T(D)>0$. However, if no more than three regions meet at every interior singular point and no more than two at every boundary singular point, then we can choose $\lambda_T(D) \le 1/4$. If in addition each of the subregions $D_k$ is convex, then we can choose any $\lambda_T(D) < 1/2$, which due to the restrictions in Lemma \ref{lem_petz} is the maximum we can achieve anyway. See for example \cite{nicaise,costabel,petzoldt_thesis} for details.

The following is an analogue of Theorem \ref{regu} on the regularity of the solution $u$ of \eqref{mod} for piecewise $C^t$ coefficients. All the other results on the finite element convergence error discussed above follow of course again from this.
\begin{theorem}\label{regu_dis} Let $D \subset \mathbb{R}^2$ be a Lipschitz polygonal domain and let $\lambda_T(D) > 0$. Suppose Assumptions A1, A2* and A3 hold with $t\leq 1$. Then, the solution $u$ of \eqref{mod} is in $L^p(\Omega, H^{1+s}(D))$, for any $0 < s < t$ such that $s \le \lambda_T(D)$ and for all $p < p_*$.
\end{theorem}
\begin{proof} Let us first consider $\phi \equiv 0$ again. Then, the existence of a unique solution $u(\omega,\cdot) \in H^1(D)$ of \eqref{mod} follows again from the Lax-Milgram Theorem, for almost all $\omega \in \Omega$. Also note that restricted to $D_k$ the transmission operator $T(\alpha_1,\ldots,\alpha_K) = \alpha_k \Delta$, for all $k=1,...,K$. Therefore, using Assumption A2* we can prove as in \S\ref{sec:poly} via a homotopy method that $u(\omega,\cdot)$ restricted to $D_k$ is in $H^{1+s}(D_k)$, for any $s < t$ and $s \le \lambda_T(D)$, for almost all $\omega \in \Omega$. The result then follows from Lemma \ref{lem_petz} and an application of H\"older's inequality. The case $\phi \not\equiv 0$ follows as in the proof to Theorem \ref{regu} via a trace estimate.
\end{proof}

As an example of a random coefficient that satisfies Assumption A2* for any $t < 1/2$, we can consider a piecewise log-normal random field $a = \exp(g)$ such that $g|_{D_k} := g_k$, for all $k=1,\ldots,K$, where each $g_k$ is a Gaussian random field with mean $\mu_k(x)$ and exponential covariance function
\[
\mathbb E\Big[(g_k(\omega,x)-\mu_k(x)])(g_k(\omega,y)-\mu_k(y)) \Big]= \sigma_k^2 \exp(-\|x-y\|/ \lambda_k).
\]
In a similar manner, if we let each $g_k$ be a Gaussian field with mean $\mu_k \in \mathcal C^1(\overline D)$ and Gaussian covariance function
\[
\mathbb E\Big[(g_k(\omega,x)-\mu_k(x)])(g_k(\omega,y)-\mu_k(y)) \Big]= \sigma_k^2 \exp(-\|x-y\|^2/ \lambda_k^2).
\]
we have Assumption A2* is satisfied for any $t\leq 1$.
The mean $\mu_k(x)$, the variance $\sigma_k^2$ and the correlation length $\lambda_k$ can be vastly different from one subregion to another.

\subsubsection{Numerics}
A rock formation which is often encountered in applications is a channelised medium. To simulate this, we divide $D$ into 3 horizontal layers, and model the permeabilities in the 3 layers by two different log--normal distributions. The middle layer, which has a higher mean permeability, occupies the region $\{1/3 \leq x_2 \leq 2/3\}$. The parameters in the top and bottom layer are taken to be $\mu_1=0$, $\lambda_1=0.3$ and $\sigma_1^2=1$, and for the middle layer we take $\mu_2=4$, $\lambda_2=0.1$ and $\sigma_2^2=1$ (assuming no correlation across layers). As a test problem we again choose the flow cell model problem \eqref{modnum2} on the unit square $D=(0,1)^2$. Samples from fields with exponential covariance are produced using the circulant embedding technique already used in \S \ref{sec:num_funcs}. Fields with Gaussian covariance are approximated by truncated Karhunen--Lo\`eve expansions. The eigenpairs of the covariance operator are computed numerically using a spectral collocation method.

Figures \ref{fig:chan_fenum} and \ref{fig:chan_fenum_gauss} show results for fields with exponential and Gaussian covariance functions, respectively. Theorem \ref{regu_dis} in both cases suggests a global spatial regularity of $H^{1/2 -\delta}(D)$, for any $\delta >0$. For fields with exponential covariance function, this is the same global regularity as in the case of continuous coefficients (satisfying A2), and convergence rates of the finite element error should not be affected by the discontinuities. For fields with Gaussian covariance function, however, continuous coefficients give a global regularity of $H^2(D)$, and so the discontinuities should lead to lower convergence rates.

The numerical results confirm this observation. For comparison, we have in Figures \ref{fig:chan_fenum} and \ref{fig:chan_fenum_gauss} added the graphs for the case where there is no ``channel'', i.e. the permeability field is one continuous log--normal field with $\mu=\mu_1=0, \, \lambda=\lambda_1=0.3$ and $\sigma^2=\sigma_1^2=1$. As expected, in Figure \ref{fig:chan_fenum}, we observe $O(h^{1/2})$ convergence of the $H^1(D)$--seminorm of the error, and linear convergence of the $L^2(D)$--norm of the error, for both permeability fields. In Figure \ref{fig:chan_fenum_gauss}, however, we indeed observe the slower convergence rates for the layered medium. Whereas we observe $O(h^{1/2})$ convergence of the $H^1(D)$--seminorm, and linear convergence of the $L^2(D)$--norm for the layered medium, we have linear convergence of the $H^1(D)$--seminorm, and quadratic convergence of the $L^2(D)$--norm for the continuous permeability field. Since the slower convergence rates are caused by singulartities at the interfaces, one could of course use local mesh refinement near the interfaces in order to recover the faster convergence rates also for the layered medium.

\begin{figure}[t]
\centering
\hspace*{-0.75cm}
\includegraphics[width=0.5\textwidth]{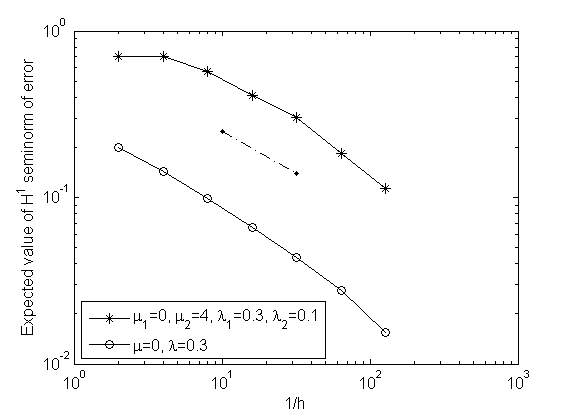}
\ \
\includegraphics[width=0.5\textwidth]{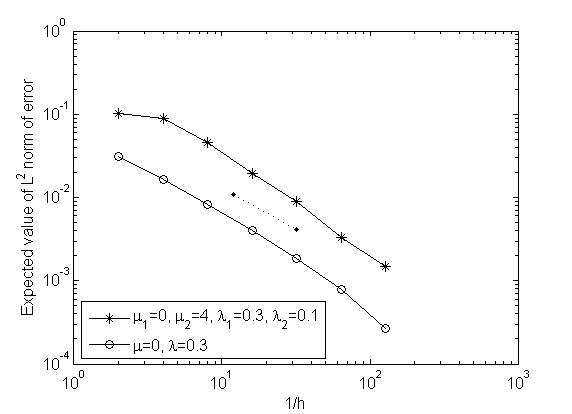}\hspace*{-0.75cm}
\caption{Left: Plot of $\EE\left[|u_{h^*} - u_{h} |_{H^1(D)}\right]$ versus $1/h$ for model problem \eqref{modnum2} with 2--norm exponential covariance, with $\mu=\mu_1=0$, $\mu_2=4$, $\lambda=\lambda_1=0.3$, $\lambda_2=0.1$, $\sigma^2=\sigma_1^2=\sigma_2^2=1$ and $h^*=1/256$. Right: Plot of $\EE\left[\|u_{h^*} - u_{h} \|_{L^2(D)}\right]$. The gradient of the dash--dotted (resp. dotted) line is $-1/2$ (resp. $-1$).\label{fig:chan_fenum}}
\end{figure}

\begin{figure}[h!]
\centering
\hspace*{-0.75cm}
\includegraphics[width=0.5\textwidth]{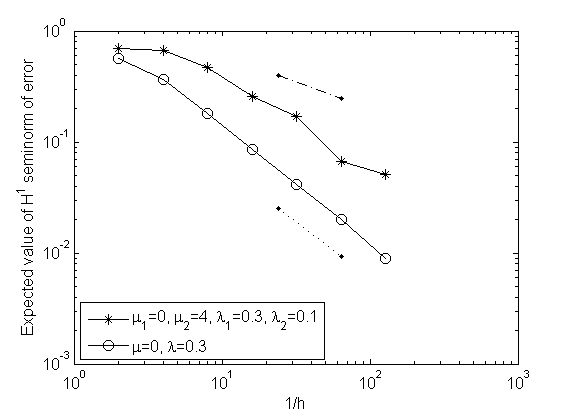}
\ \
\includegraphics[width=0.5\textwidth]{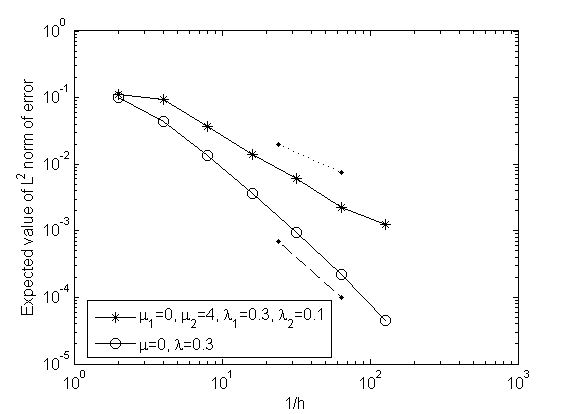}\hspace*{-0.75cm}
\caption{Left: Plot of $\EE\left[|u_{h^*} - u_{h} |_{H^1(D)}\right]$ versus $1/h$ for model problem \eqref{modnum2} with Gaussian covariance, with $\mu=\mu_1=0$, $\mu_2=4$, $\lambda=\lambda_1=0.3$, $\lambda_2=0.1$, $\sigma^2=\sigma_1^2=\sigma_2^2=1$, $h^*=1/256$ and $K^*=170$. Right: Plot of $\EE\left[\|u_{h^*} - u_{h} \|_{L^2(D)}\right]$. The gradient of the dash--dotted (resp. dotted and dashed) line is $-1/2$ (resp. $-1$ and$-2$).\label{fig:chan_fenum_gauss}}
\end{figure}

\section{Conclusions and Further Work}
Multilevel Monte Carlo methods have the potential to significantly outperform standard Monte Carlo simulations in a variety of contexts. In this paper, we considered the application of multilevel Monte Carlo methods to elliptic PDEs with random coefficients, in the practically relevant and technically more demanding case of log--normal random coefficients with short correlation lengths where realisations of the diffusion coefficient have limited regularity and are not uniformly bounded or elliptic. We extended the theory from \cite{cst11} to cover more difficult model problems, including corner domains and discontinuous means, and we offered one possible remedy for the problem of correlation length dependent coarse mesh size restrictions in the standard multilevel estimator. This was done by using level--dependent truncations of the Karhunen-Lo\`eve expansion of the coefficient, resulting in smoother approximations of the coefficient on the coarser levels.

An issue we plan to investigate further in the future, is how to achieve smoother approximations of the random coefficient on the coarser grids also using other sampling techniques, such as in the circulant embedding method used in \S \ref{sec:func}. Another area of future research is the adaptive choice of spatial grids in the multilevel estimator.

\bibliographystyle{plain}
\bibliography{bibMLMC}
\end{document}